%% file: monotonepathsmixslowly.tex
\documentclass[11pt]{articlefederico}
\usepackage{amsmath,amssymb,amsfonts}
\usepackage{amsthm}
\usepackage{enumerate}
\usepackage{young}
\usepackage{xcolor}
\usepackage{graphicx}
\usepackage{epsfig}
\usepackage{epstopdf}
\usepackage{color}
\usepackage{float}
\usepackage[a4paper]{geometry}
\usepackage{ytableau}
\usepackage{mathtools}
\usepackage{hyperref}
\usepackage[none]{hyphenat}
\usepackage{lscape}
\usepackage{verbatim}
\usepackage{wasysym}
\usepackage{tikz}
\usetikzlibrary{arrows, matrix, fit}
\usepackage{mathtools}
\usepackage{colortbl}
\newcommand{\defeq}{\vcentcolon=}

\makeatletter
\newif\if@borderstar
\def\bordermatrix{\@ifnextchar*{%
\@borderstartrue\@bordermatrix@i}{\@borderstarfalse\@bordermatrix@i*}%
}
\def\@bordermatrix@i*{\@ifnextchar[{\@bordermatrix@ii}{\@bordermatrix@ii[()]}}
\def\@bordermatrix@ii[#1]#2{%
\begingroup
\m@th\@tempdima8.75\p@\setbox\z@\vbox{%
\def\cr{\crcr\noalign{\kern 2\p@\global\let\cr\endline }}%
\ialign {$##$\hfil\kern 2\p@\kern\@tempdima & \thinspace %
\hfil $##$\hfil && \quad\hfil $##$\hfil\crcr\omit\strut %
\hfil\crcr\noalign{\kern -\baselineskip}#2\crcr\omit %
\strut\cr}}%
\setbox\tw@\vbox{\unvcopy\z@\global\setbox\@ne\lastbox}%
\setbox\tw@\hbox{\unhbox\@ne\unskip\global\setbox\@ne\lastbox}%
\setbox\tw@\hbox{%
$\kern\wd\@ne\kern -\@tempdima\left\@firstoftwo#1%
\if@borderstar\kern2pt\else\kern -\wd\@ne\fi%
\global\setbox\@ne\vbox{\box\@ne\if@borderstar\else\kern 2\p@\fi}%
\vcenter{\if@borderstar\else\kern -\ht\@ne\fi%
\unvbox\z@\kern-\if@borderstar2\fi\baselineskip}%
\if@borderstar\kern-2\@tempdima\kern2\p@\else\,\fi\right\@secondoftwo#1 $%
}\null \;\vbox{\kern\ht\@ne\box\tw@}%
\endgroup
}
\makeatother

\newtheorem{theorem}{Theorem}[section]
\newtheorem{corollary}[theorem]{Corollary}
\newtheorem{idea}[theorem]{Idea}
\newtheorem{lemma}[theorem]{Lemma}
\newtheorem{remark}[theorem]{Remark}

\newtheorem{proposition}[theorem]{Proposition}

\newtheorem{definition}[theorem]{Definition}

\newtheorem{propdef}[theorem]{Proposition/Definition}

\renewcommand{\S}{\mathcal{S}}

\usepackage{xcolor}

\definecolor{darkgreen}{rgb}{0,.5,0}
\definecolor{brown}{rgb}{0.5,0.3,0}

\begin{document}

\title{\textsf{Markov chains, CAT(0) cube complexes, and enumeration: \\ monotone paths in a strip mix slowly}}

\author{
\textsf{Federico Ardila--Mantilla}\footnote{\noindent \textsf{Department of Mathematics, San Francisco State University, Departamento de Matem\'aticas, Universidad de Los Andes; \texttt{federico@sfsu.edu}. Supported by National Science Foundation grant DMS-2154279.}}
\and
\textsf{Naya Banerjee}\footnote{\noindent \textsf{Department of Mathematical Sciences, University of Delaware; \texttt{banerjee.naya@gmail.com}. Work supported by National Science Foundation grant DMS-1554783.}}
\and
\textsf{Coleson Weir}\footnote{\noindent \textsf{Dept. of Economics, University of Notre Dame; \texttt{cweir2@nd.edu}. Research supported by National Science Foundation grant DMS-1554783.}} 
}
\date{}

\maketitle

\begin{abstract}
%This paper lies at the intersection of enumerative combinatorics, geometric group theory, and probability theory. We begin by studying the number $c_m(n)$ of monotone paths of length $n$ in a strip of width $m$. We find exponential growth constant of $c_m(n)$ for arbitrary $m$, generalizing results of Williams for $m=2, 3$.
%
%In the second part of the paper we use the theory of non-positively curved (CAT(0)) cubical complexes can be used to detect small bottlenecks in many graphs of combinatorial interest. This can be used to show that natural Markov chains mix slowly. We illustrate these techniques in detail for two  natural Markov chains on the set of monotone paths in a strip mix slowly. Our proof relies on the enumeration of monotone paths in a strip of any width $m$; we find the connective constants for these problems, , generalizing results of Williams for $m=2, 3.$
We prove that two natural Markov chains on the set of monotone paths in a strip mix slowly. To do so, we make novel use of the theory of non-positively curved (CAT(0)) cubical complexes to detect small bottlenecks in many graphs of combinatorial interest. Along the way, we give a formula for the number $c_m(n)$ of monotone paths of length $n$ in a strip of height $m$. In particular we compute the exponential growth constant of $c_m(n)$ for arbitrary $m$, generalizing results of Williams for $m=2, 3$.
\end{abstract}

\section{\textsf{Introduction}} \label{sec:intro}
This paper uses tools from geometric group theory and enumerative combinatorics to derive probabilistic consequences about random walks of a combinatorial nature. Our methods have wide applicability, but we focus on one example of interest, which we carry out in detail: a random walk on the set of monotone paths in a strip.

A \emph{monotone path of length $n$  in a strip of height $m$} is a lattice paths that start at $(0,0)$, takes steps $N =(0,1)$, $S=(0,-1)$ and $E=(1,0)$, never retraces steps, and stays within the strip, as shown in Figure \ref{fig:arm} for $n=15$ and $m=3$.

%We consider a robotic arm of length $n$ having $n$ links of unit length moving in a rectangular tunnel in a two dimensional grid of height $m$ units. As shown in the figure below, the robotic arm is pinned to the lower left corner of the grid and each link can go up, down, or right; it can never go left or intersect itself. Figure \ref{fig:arm} shows a robotic arm configuration with 15 unit length arms in a height 3 tunnel.

\begin{figure}[h]
\begin{center}
\begin{tikzpicture}
\draw[step=.5cm,gray] (-5,-1.5) grid (0,0);
\draw[ultra thick, blue, ->] (-5,-1.5) -- (-4,-1.5) -- (-4,-1) -- (-3.5,-1) -- (-3.5,-.5) -- (-3,-.5) -- (-3,0) -- (-2.5,0) -- (-2,0) -- (-2, -.5) -- (-1.5,-.5) -- (-1.5,-1) -- (-1,-1) -- (-1, -.5) -- (-1,0); 
%\draw (-5,-1.5) circle[radius=2pt, draw=blue, fill=blue];
%\fill (-5,-1.5) circle[radius=2pt, draw=blue, fill=blue];
 \node [draw=blue,fill=blue,circle, inner sep=2pt,minimum size=2pt] at (-5,-1.5) {};
\end{tikzpicture} 
\caption{A monotone path of length 15 in a strip of height 3.\label{fig:arm}} 
\end{center}
\end{figure}
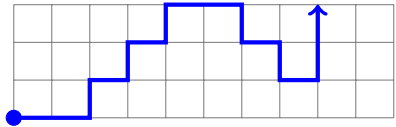

All monotone paths can be connected to each other by two kinds of local moves, illustrated in Figure \ref{fig:switch}.

\noindent
$\bullet$  switch corners: two consecutive steps that go in different directions exchange directions.

\noindent
$\bullet$  flip the end: the last step of the path rotates $90^\circ$.

%\begin{itemize}
%\item switching corners: two consecutive steps that go in different directions exchange directions.
%\item flipping the end: the last step of the path rotates $90^\circ$.
%%\item Switch corners: two consecutive links going in different directions exchange directions.
%%\item Flip the end: the last link of the robot rotates $90^\circ$.
%\end{itemize}

\begin{figure}[h]\label{fig:moves}
\begin{center}
\begin{tikzpicture}
\draw[step=0.5cm,gray] (-0.5,-0.5) grid (0,0);
\draw[ultra thick, -, blue]  (-0.5,-0.5) -- (-0.5,0) -- (0,0);
\end{tikzpicture}
\begin{tikzpicture}
\draw[step=0.5cm,white,line width=0mm] (-0.5,-0.5) grid (0,0);
\draw[thick, <->] (-.5,-.25) -- (0,-.25);
\end{tikzpicture}
\begin{tikzpicture}
\draw[step=0.5cm,gray] (-0.5,-0.5) grid (0,0);
\draw[ultra thick, -, blue] (-0.5,-0.5) -- (0,-0.5) -- (0,0);
\end{tikzpicture}
\qquad  
\begin{tikzpicture}
\draw[step=0.5cm,gray] (-0.5,-0.5) grid (0,0);
\draw[ultra thick, -, blue]  (-0.5,0) -- (-0.5,-0.5) -- (0,-0.5);
\end{tikzpicture}
\begin{tikzpicture}
\draw[step=0.5cm,white,line width=0mm] (-0.5,-0.5) grid (0,0);
\draw[thick, <->] (-.5,-.25) -- (0,-.25);
\end{tikzpicture}
\begin{tikzpicture}
\draw[step=0.5cm,gray] (-0.5,-0.5) grid (0,0);
\draw[ultra thick, -, blue] (-0.5,0) -- (0,0) -- (0,-0.5);
\end{tikzpicture}
\qquad
\qquad
\begin{tikzpicture}
\draw[step=0.5cm,gray] (-0.5,-0.5) grid (0,0);
\draw[ultra thick, ->, blue]  (-0.5,-0.5) -- (-0.5,0);
\end{tikzpicture}
\begin{tikzpicture}
\draw[step=0.5cm,white,line width=0mm] (-0.5,-0.5) grid (0,0);
\draw[thick, <->] (-0.5,-.25) -- (0,-.25);
\end{tikzpicture}
\begin{tikzpicture}
\draw[step=0.5cm,gray] (-0.5,-0.5) grid (0,0);
\draw[ultra thick, ->, blue] (-0.5,-0.5) -- (0,-0.5);
\end{tikzpicture}
\qquad 
\label{Figure3a}
\begin{tikzpicture}
\draw[step=0.5cm,gray] (-0.5,-0.5) grid (0,0);
\draw[ultra thick, ->, blue]  (-0.5,0) -- (-0.5,-0.5);
\end{tikzpicture}
\begin{tikzpicture}
\draw[step=0.5cm,white,line width=0mm] (-0.5,-0.5) grid (0,0);
\draw[thick, <->] (-0.5,-.25) -- (0,-.25);
\end{tikzpicture}
\begin{tikzpicture}
\draw[step=0.5cm,gray] (-0.5,-0.5) grid (0,0);
\draw[ultra thick, ->, blue] (-0.5,0) -- (0,0);
\end{tikzpicture}
%\end{center} 
%\caption{Flipping the end.\label{fig:flip}}
%\end{figure}

\end{center}
\caption{(a) Switching corners. \qquad \qquad \qquad \qquad \qquad    (b) Flipping the end. \label{fig:switch}}
\end{figure}
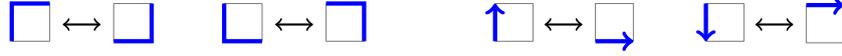

This model was introduced by Abrams and Ghrist \cite{Abr04}, who thought of this as a model for a robotic arm in a tunnel. They showed that the graph $S_{m,n}$ of possible configurations of the robotic arm connected by these local moves is a \emph{cubical complex} $\S_{m,n}$: a polyhedral complex made of unit cubes of varying dimensions. Ardila, Bastidas, Ceballos, and Guo {\cite{Ard17} proved that the space $\S_{m,n}$ is non-positively curved, or \emph{CAT(0)}, and used this to give an algorithm to move the robotic arm optimally between any two configurations. They also derived a  formula for the diameter of the graph $S_{m,n}$.

\bigskip

We revisit this combinatorial model, with new goals in mind. A central idea, relying on Ardila, Owen, and Sullivant's one-to-one correspondence between CAT(0) cube complexes and posets with inconsistent pairs (PIPs) \cite{AOS}, is the following:

\begin{idea} \label{idea}
When the transition kernel of a Markov chain $M$ is a CAT(0) cube complex, one can use the corresponding poset with inconsistent pairs (PIP) $P_M$ to find bottlenecks (vertex separators) in the kernel, and obtain upper bounds on the mixing time of $M$.
\end{idea}

This idea applies quite generally. In this paper we illustrate it in a detailed example: two random walks on the set of monotone paths in a strip. We make three main contributions, which we now describe. For precise definitions and statements, we refer the reader to the corresponding sections of the paper.

\bigskip

\textbf{A.} In Section \ref{sec:enum} we study the number $c_m(n)$ of monotone paths of length $n$ in a strip of height $m$. The generating functions for $c_2(n)$ and $c_3(n)$ were first computed by Williams  \cite{Wil96}. We give a general formula for the generating function of $c_m(n)$ for any height $m$. In particular, we are able to compute the \emph{exponential growth constant} or \emph{connective constant} $r_m=\lim_{n \rightarrow \infty} \sqrt[n]{c_m(n)}$ for any $m$. 

\begin{theorem}  \label{thm:main1} For each $m \geq 0$ there are constants $q_m$ and $r_m$ such that
\[
c_m(n) \sim q_m \cdot r_m^n,
\]
where
$ \displaystyle
1=r_0 < r_1 < r_2 < \cdots, \mathrm{ \ and \ } \lim_{m \rightarrow \infty} r_m = 1+\sqrt{2} \approx 2.4142\ldots.
$
The growth constant $r_m$ is the largest real root of the polynomial
%
%\begin{eqnarray*}
%a_0(x) &=& 1-x \\x
%a_1(x) &=&  x^4-2x^3+x^2-1 \\
%a_m(x) &=& (-x^3+x^2-x-1) a_{m-1}(x) - x^4 a_{m-2}(x) \qquad \text{ for } m \geq 2\\
%\end{eqnarray*}
\[
a_m(x) = \alpha_+(x) \beta_+(x)^m + \alpha_-(x) \beta_-(x)^m,
\]
where 
\begin{eqnarray*}
\alpha_\pm(x) &=& \pm \frac{x^4-2x^3-1}{2 \sqrt{(x^4-1)(x^2-2x-1)}} + \frac{1-x}2, \\
\beta_\pm(x) &=& \frac{-x^3+x^2-x-1 \pm \sqrt{(x^4-1)(x^2-2x-1)}}{2}.  
\end{eqnarray*}
\end{theorem}

This theorem allows us to use computer software to easily compute $r_m$ for concrete values of $m$.
This description of the growth constant $r_m$ is the most explicit possible, because Galois theory tells us that there is no exact formula for it.
 
\bigskip

\textbf{B.} In Section \ref{sec:CAT(0)} we find a small bottleneck in the transition kernel $S_{m,n}$ of monotone paths, connected by the local moves described above. We prove:

\begin{theorem} \label{thm:main2} For each $m \geq 2$, the transition kernel $S_{m,n}$ has a small bottleneck of $n+1$ vertices, whose removal separates the graph into two components $(S_{m,n})_a$ and $(S_{m,n})_b$ of sizes
\[
|(S_{m,n})_a| \sim (1-C_m)|S_{m,n}| \qquad |(S_{m,n})_b| \sim C_m|S_{m,n}|
\]
where $C_m = 1/(r_m^2(r_m-1)^2)$. These constants satisfy
$
0.345\ldots \approx C_2 > C_3 > C_4 \geq \cdots$ and $\displaystyle \lim_{m \rightarrow \infty} C_m = 1.5-\sqrt{2} \approx 0.0858\ldots.
$
\end{theorem}
%
%Our methods apply to many situations, when the transition graph is the skeleton of a CAT(0) (non-positively curved) cube complex. When that is the case, Ardila, Bastidas, Ceballos, and Guo showed that we can encode the cube complex -- and hence the graph -- in a \emph{poset with inconsistennt pairs} or \emph{PIP}. Here we show 
%
%\begin{idea} \label{idea:main}
%When a graph is the $1$-skeleton of a CAT(0) cube complex $X$, we can use the corresponding PIP $P_X$ to find bottlenecks in the transition graph.
%\end{idea}

\bigskip

\textbf{C.} In Section \ref{sec:Markov} we use the results of A. and B. above to show that two natural Markov chains on the space of monotone paths, which we call the ``symmetric" and ``lazy simple" Markov chains, mix slowly. In each step of the symmetric Markov chain $M_{m,n}$, we choose a vertex of the path uniformly at random, and perform a local move there if it is available. In each step of the lazy simple Markov chain $N_{m,n}$, at each step we first decide whether to move with probability $p$, and if we do, we perform one of the available moves uniformly at random.

\begin{theorem} \label{thm:main3}
Let $m \geq 2$. For the symmetric Markov chain $M_{m,n}$, the stationary distribution is uniform, and the mixing time grows exponentially with $n$:
\[
\tau_{M_{m,n}}(\varepsilon)  \ge C \cdot r_m^n \ln(\varepsilon^{-1})
\]
For the lazy simple Markov chain $N^p_{m,n}$, the stationary distribution is proportional to the degree -- with $\pi(x) \propto \deg(x)$ for each lattice path $x$ -- and the mixing time grows exponentially with $n$:
\[
\tau_{N^p_{m,n}}(\varepsilon)  \ge D \cdot \frac1n r_m^n \ln(\varepsilon^{-1}).
\]
Here $0<\varepsilon<1$ is arbitrary, $C$ and $D$ are constants, and $r_m>1$ is the constant of Theorem \ref{thm:main1}.
\end{theorem}

\section{\textsf{Enumeration of monotone paths in a strip}}\label{sec:enum}

Let $S_m$ be the strip of height $m$ that extends infinitely to the right:
\[
S_m = \{(x,y) \, | \, x \geq 0, \, 0 \leq y \leq m\}.
\]

\begin{definition}
A \emph{monotone path in a strip} is a lattice path that starts at $(0,0)$, takes steps $N =(0,1)$, $S=(0,-1)$ and $E=(1,0)$, never retraces steps, and stays within the strip.
\end{definition}

In this section we study the enumeration of monotone paths in a strip. 

\subsection{\textsf{The transfer-matrix method: monotone paths as walks in a graph}}

To enumerate monotone paths in a strip, we use the \emph{transfer-matrix method}, which we now briefly recall. For a thorough treatment of the transfer-matrix method, including the relevant proofs, see for example \cite[Chapter 4.7]{Sta11}

Suppose we are interested in a certain family of combinatorial objects, and we wish to find the number $a_n$ of objects of ``size" $n$. The idea is to construct a directed graph $G=(V,E)$ such that the objects of size $n$ can be encoded as -- that is, they are in bijection with -- the walks of length $n$ in $G$. If we succeed, then the enumeration problem becomes a linear algebra problem. If we let $A$ be the $V \times V$ adjacency matrix of $G$, given by
\[
A_{uv} = \text{ number of edges from $u$ to $v$} \qquad \text{ for } u,v \in V
\]
then the key observation of the transfer-matrix method  \cite[Theorem 4.7.1]{Sta11} is that
\[
A^n_{uv} = \text{ number of paths of length $n$ from $u$ to $v$} \qquad \text{ for } u,v \in V.
\]
Our enumeration problem then reduces to computing powers of the adjacency matrix. In particular, the eigenvalues of $A^0,A^1,A^2,\ldots$ control the growth of the sequence $a_0, a_1, a_2, \ldots$.

Let us now apply this philosophy to the problem that interests us.

\bigskip

\noindent
\textbf{A failed encoding of monotone paths as walks in a graph.}
A natural first idea is to encode a path by keeping track of the height $h_i$ of each node $i$. For example, the monotone path in Figure \ref{fig:arm} has node heights $0,0,0,1,1,2,2,3,3,3,2,2,1,1,2,3$. This sequence of heights can be viewed as a walk in the graph $F_m$ with vertices $0,1,\ldots,m$ and edges $i \rightarrow i-1$,\,  $i \rightarrow i$,\,  $i \rightarrow i+1$ whenever those vertices are in the range $\{0, \ldots, m\}$. Each monotone paths gives rise to a different walk in this graph $F_m$ starting at 0. However, not every such walk arises from a monotone path: a monotone path cannot contain an edge $i \rightarrow i+1$ (resp. $i \rightarrow i-1$) followed by the reverse edge $i+1 \rightarrow i$ (resp. $i-1 \rightarrow i$). For this reason, the graph $F_m$ does not provide the necessary bijective encoding.

\bigskip

\noindent
\textbf{A successful encoding of monotone paths as walks in a graph.} To resolve the issue above, we redundantly insert some memory into the bookkeeping procedure. We encode a path as a sequence of pairs: 
for the $i$th node, we keep track of the \emph{height pair} $(h_{i-1}, h_i)$ of heights of nodes $i-1$ and $i$. (We define $h_{-1}=0$.) 
For example, the monotone path in Figure \ref{fig:arm} is recorded by the successive height pairs $00, 00,  01, 11, 12, 22, 23, 33, 33, 32, 22, 21, 11, 12, 23$.\footnote{We write $ij$ for the pair $(i,j)$ when it introduces no confusion.}

Now we consider the directed graph whose vertices are the possible height pairs, and whose edges are the possible transitions between two consecutive height pairs.

\begin{definition} Let $m$ be a positive integer. The \emph{transfer graph} $G_m$ is the directed graph with

%$\bullet$ vertices: $(i,i)$ for $0 \leq i \leq m$,  $(i,i+1)$ for $0 \leq i \leq m-1$, and  $(i, i-1)$ for $1 \leq i \leq m$ 

$\bullet$ vertices: $(i, i-1), \, (i,i), \, (i,i+1)$ whenever the indices are between $0$ and $m$, and

$\bullet$ edges: $(i,j) \rightarrow (j,k)$ unless $j=i \pm 1$ and $k=i$.

\end{definition}

The graph $G_2$ is illustrated in Figure \ref{fig:graph}.

\begin{figure}[h]
\begin{minipage}{0.6\textwidth}
\begin{center}
\resizebox{0.65\textwidth}{!}{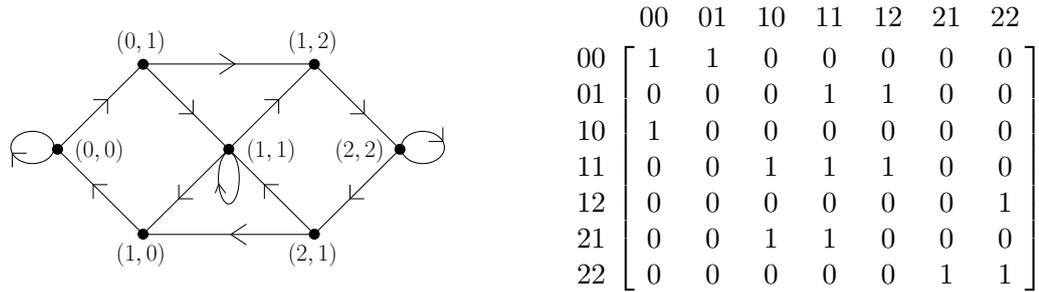} 

\end{center}
\end{minipage}
\begin{minipage}{0.25\textwidth}
%\begin{eqnarray*}
$
\bordermatrix[{[]}]{%
  & 00 & 01 & 10 & 11 & 12 & 21 & 22 \cr
00 &  1 & 1 & 0 & 0 & 0 & 0 & 0 \cr
01 &0 & 0 & 0 & 1 & 1 & 0 & 0 \cr
10 &1 & 0 & 0 & 0 & 0 & 0 & 0 \cr
11 & 0 & 0 & 1 & 1 & 1 & 0 & 0 \cr
12 & 0 & 0 & 0 & 0 & 0 & 0 & 1 \cr
21 & 0 & 0 & 1 & 1 & 0 & 0 & 0 \cr
22 & 0 & 0 & 0 & 0 & 0 & 1 & 1 }
$
%A_2 = \begin{bmatrix}
%
%  1 & 1 & 0 & 0 & 0 & 0 & 0 \\
%0 & 0 & 0 & 1 & 1 & 0 & 0 \\
%1 & 0 & 0 & 0 & 0 & 0 & 0 \\
% 0 & 0 & 1 & 1 & 1 & 0 & 0 \\
% 0 & 0 & 0 & 0 & 0 & 0 & 1 \\
% 0 & 0 & 1 & 1 & 0 & 0 & 0 \\
% 0 & 0 & 0 & 0 & 0 & 1 & 1 \\
%\end{bmatrix}
%\end{eqnarray*}
\end{minipage}
\caption{The graph $G_2$ and its adjacency matrix $A_2$.\label{fig:graph}.}
\end{figure}

\begin{lemma}\label{lemma:transfermatrix}
The number $c_m(n)$ of monotone paths of length $n$ in a strip of height $m$ equals the number of walks of length $n$ in the transfer graph $G_m$ that start at vertex $(0,0)$.
\end{lemma}

%\begin{lemma}
%The number of configurations $c_m(n) \defeq |\Omega_{m,n}|$ of the arm is equal to the number of walks in $G_m$ of length $n$ starting at vertex $00$.
%\end{lemma}

\begin{proof}
For a monotone path of length $n$, let $0=h_0, h_1, h_2, \ldots, h_n$ be the heights of the nodes. Since the path cannot retrace steps, the sequence above cannot contain a consecutive subsequence $h,h+1,h$ or $h+1,h,h+1$. Therefore $(0,0) \rightarrow (0,h_1) \rightarrow (h_1,h_2) \rightarrow \cdots \rightarrow (h_{n-1},h_n)$ is a walk of length $n$ in $G_m$. Conversely, every walk of length $n$ starting at $(0,0)$ corresponds to such a monotone path.
\end{proof}

\subsection{\textsf{The characteristic polynomial}}

Having established the encoding of monotone paths as walks in a transfer graph $G_m$, we proceed with the linear algebraic analysis. 
Let $A_m$ be the $(3m+1) \times (3m+1)$ adjacency matrix of $G_m$; its non-zero entries are:
\[
(A_m)_{(i,j),(j,k)} = 1 \text{ unless } j = i \pm1 \text{ and } i=k
\]
all other entries equal 0.
Let its characteristic polynomial be
%$ iff there is an edge from $(i,j)$ to $(j,k)$. 
\[
a_m(x) = \det(A_m - xI).
\]
The adjacency matrix $A_2$ is shown in Figure \ref{fig:graph}, and its characteristic polynomial is $a_2(x) = -x^7+3x^6-3x^5+x^4+2x^3-2x^2+x+1$.

\begin{lemma} \label{lem:rec}
The characteristic polynomial $a_m(x) = \det(A_m - xI) $ of the adjacency matrix of the graph $G_m$ is given by the recurrence 
%the recurrence
%\[
%a_m(x) = (-x^3+x^2-x-1) a_{m-1}(x) - x^4 a_{m-2}(x) \qquad \text{ for } m\geq 2
%\]
%with initual values
%\[
%a_0(x) = 1-x, \qquad a_1(x) =  x^4-2x^3+x^2-1.
%\]
%
\[
a_m(x) = \begin{cases}
1-x & m=0 \\
x^4-2x^3+x^2-1 & m=1 \\
(-x^3+x^2-x-1) a_{m-1}(x) - x^4 a_{m-2}(x) &  m \geq 2.
\end{cases}
\]
%\end{eqnarray*}

%
%
%
%\begin{eqnarray*}
%
%\left(
%\frac{-2x^8+4x^7-2x^6+3x^4-2x^3+2x^2-1}{2 \sqrt{x^6-2x^5-x^4-x^2+2x-1}} + \frac{1-x}2
%\right)
%\left(\frac{-x^3+x^2-x-1 + \sqrt{x^6-2x^5-x^4-x^2+2x-1}}{2x^4} \right)^n \\
%&+&  
%(mess)\left(\frac{-x^3+x^2-x-1 - \sqrt{x^6-2x^5-x^4-x^2+2x-1}}{2x^4} \right)^n
%\end{eqnarray*}
\end{lemma}

\begin{proof}
Consider the $3 \times 3$ matrices 
\[
X = 
\begin{bmatrix}
1 & 1 & 0 \\
0 & 0 & 0 \\
0 & 0 & 0 
\end{bmatrix},
\qquad 
Y = 
\begin{bmatrix}
-x & 0 & 0 \\
1 & 1-x & 1 \\
0 & 0 & -x 
\end{bmatrix}, 
\qquad 
Z = 
\begin{bmatrix}
0 & 0 & 0 \\
0 & 0 & 0 \\
0 & 1 & 1 
\end{bmatrix}
\]

For a matrix $S$, let $\overline S_{i,}, \overline S_{,j}$ and $\overline S_{i,j}$ respectively denote the matrix $S$ with its $i$-th row, $j$-th column or the $i$-th row and $j$-th column deleted. 
Then one can verify that $A_m-xI$ is the $(3m+1) \times (3m+1)$ matrix with block format given below.
\[
A_m-xI = 
\begin{bmatrix}
\overline Y_{1,1} & \overline Z_{1,} & 0 & 0 & \cdots & 0 & 0 & 0 \\
\overline X_{,1} & Y & Z & 0 & \cdots & 0 & 0 & 0 \\
0 & X & Y & Z & \cdots & 0 & 0 & 0 \\
\vdots & \vdots & \vdots & \vdots & \ddots & \vdots & \vdots & \vdots \\
0 & 0 & 0 & 0 & \cdots & Y & Z & 0\\
0 & 0 & 0 & 0 & \cdots & X & Y & \overline Z_{,3} \\
0 & 0 & 0 & 0 & \cdots & 0 & \overline X_{3,} & \overline Y_{3,3} \\

\end{bmatrix}
\]
Expanding by cofactors,

%\pgfdeclarelayer{background}
%\pgfsetlayers{background,main}
%
%\begin{tikzpicture}
%\matrix[matrix of math nodes,left delimiter = (,right delimiter = ),row sep=10pt,column sep = 10pt] (m)
% {
%1-x & 1 & 0 &0 & 0 &0& \cdots \\
%0 & -x & 0 &1 &1& 0 & \cdots \\
%1 & 0 & -x &0 &0 &0 & \cdots\\
%0 & 0 & 1& 1-x&1 & 0 &\cdots \\
%0 & 0 &0 & 0&-x &0 & \cdots \\
%0 & 0 &0 & 0&0 &-x & \cdots \\
%\vdots & \vdots &\vdots & \vdots& \vdots &\vdots & \ddots \\
% };
% \begin{pgfonlayer}{background}
% \node[inner sep=3pt,fit=(m-1-1)]          (1)   {};
% \node[inner sep=3pt,fit=(m-1-7)]          (2)   {};
% \node[inner sep=3pt,fit=(m-7-1)]	   (3)   {};
% \fill[gray, opacity=0.2] (1.north west) rectangle (2.south east);
% \fill[gray, opacity=0.2] (1.north west) rectangle (3.south east);
%\end{pgfonlayer}
%\end{tikzpicture}

\begin{eqnarray*}
a_m(x) &=& \det (A_m-xI) \\
&=& 
\left|\,\begin{array}{ccccccc}
1-x &1 & 0 &0 & 0 &0& \cdots \\
0 & -x & 0 &1 &1& 0 & \cdots \\
1 & 0 & -x &0 &0 &0 & \cdots\\
0 & 0 & 1& 1-x&1 & 0 &\cdots \\
0 & 0 &0 & 0&-x &0 & \cdots \\
0&0&0&0&0&0& \cdots\\
\vdots & \vdots &\vdots & \vdots& \vdots &\vdots & \ddots \\
\end{array}\,\right| \\
&=& 
(1-x) 
\left|\,\begin{array}{ccccccc}
\cellcolor{gray}&\cellcolor{lightgray} & \cellcolor{lightgray} & \cellcolor{lightgray} & \cellcolor{lightgray}&\cellcolor{lightgray}& \cellcolor{lightgray} \cdots \\
\cellcolor{lightgray} & -x & 0 &1 &1& 0 & \cdots \\
\cellcolor{lightgray}& 0 & -x &0 &0 &0 & \cdots\\
\cellcolor{lightgray} & 0 & 1& 1-x&1 & 0 &\cdots \\
\cellcolor{lightgray}& 0 &0 & 0&-x &0 & \cdots \\
\cellcolor{lightgray}&0&0&0&0&0& \cdots\\
\cellcolor{lightgray} \vdots & \vdots &\vdots & \vdots& \vdots &\vdots &\ddots \\
\end{array}\,\right|
-
(1) 
\left|\,\begin{array}{ccccccc}
\cellcolor{lightgray} &\cellcolor{gray} & \cellcolor{lightgray} &\cellcolor{lightgray}& \cellcolor{lightgray} &\cellcolor{lightgray}& \cellcolor{lightgray}\cdots \\
0 & \cellcolor{lightgray}  & 0 &1 &1& 0 & \cdots \\
1 & \cellcolor{lightgray}& -x &0 &0 &0 & \cdots\\
0 & \cellcolor{lightgray}& 1& 1-x&1 & 0 &\cdots \\
0 & \cellcolor{lightgray}&0 & 0&-x &0 & \cdots \\
0&\cellcolor{lightgray}&0&0&0&0& \cdots\\
\vdots &\cellcolor{lightgray} \vdots &\vdots & \vdots& \vdots &\vdots & \ddots \\
\end{array}\,\right| \\
&=& (1-x)(-x) 
\left|\,\begin{array}{ccccccc}
\cellcolor{gray} &\cellcolor{lightgray} & \cellcolor{lightgray} & \cellcolor{lightgray} & \cellcolor{lightgray} &\cellcolor{lightgray}& \cellcolor{lightgray} \cdots \\
\cellcolor{lightgray} & \cellcolor{gray}  & \cellcolor{lightgray} &\cellcolor{lightgray} &\cellcolor{lightgray}& \cellcolor{lightgray}&\cellcolor{lightgray} \cdots \\
\cellcolor{lightgray} & \cellcolor{lightgray} & -x &0 &0 &0 & \cdots\\
\cellcolor{lightgray} & \cellcolor{lightgray} & 1& 1-x&1 & 0 &\cdots \\
\cellcolor{lightgray} & \cellcolor{lightgray} &0 & 0&-x &0 & \cdots \\
\cellcolor{lightgray}&\cellcolor{lightgray}&0&0&0&0& \cdots\\
\cellcolor{lightgray} \vdots & \cellcolor{lightgray}\vdots &\vdots & \vdots& \vdots &\vdots & \ddots\\
\end{array}\,\right| 
-
(1)(-1) 
\left|\,\begin{array}{ccccccc}
\cellcolor{lightgray} &\cellcolor{gray} & \cellcolor{lightgray}&\cellcolor{lightgray} & \cellcolor{lightgray} &\cellcolor{lightgray}& \cellcolor{lightgray}\cdots \\
\cellcolor{lightgray} & \cellcolor{lightgray}  & 0 &1 &1& 0 & \cdots \\
\cellcolor{gray} & \cellcolor{lightgray} & \cellcolor{lightgray} -x &\cellcolor{lightgray} &\cellcolor{lightgray} &\cellcolor{lightgray}0 & \cellcolor{lightgray}\cdots\\
\cellcolor{lightgray} & \cellcolor{lightgray} & 1& 1-x&1 & 0 &\cdots \\
\cellcolor{lightgray} & \cellcolor{lightgray} &0 & 0&-x &0 & \cdots \\
\cellcolor{lightgray}&\cellcolor{lightgray}&0&0&0&0& \cdots\\
\cellcolor{lightgray} \vdots &\cellcolor{lightgray} \vdots &\vdots & \vdots& \vdots &\vdots & \ddots\\
\end{array}\,\right|  
\end{eqnarray*}

\begin{eqnarray*}
&=& (1-x)(-x)(-x) 
\left|\,\begin{array}{ccccccc}
\cellcolor{gray} &\cellcolor{lightgray} & \cellcolor{lightgray} & \cellcolor{lightgray} & \cellcolor{lightgray} &\cellcolor{lightgray}& \cellcolor{lightgray} \cdots \\
\cellcolor{lightgray} & \cellcolor{gray}  & \cellcolor{lightgray} &\cellcolor{lightgray} &\cellcolor{lightgray}& \cellcolor{lightgray} &\cellcolor{lightgray} \cdots \\
\cellcolor{lightgray} & \cellcolor{lightgray} & \cellcolor{gray} &\cellcolor{lightgray} &\cellcolor{lightgray} &\cellcolor{lightgray} & \cellcolor{lightgray}\cdots\\
\cellcolor{lightgray} & \cellcolor{lightgray} & \cellcolor{lightgray}& 1-x&1 & 0 &\cdots \\
\cellcolor{lightgray} & \cellcolor{lightgray} &\cellcolor{lightgray} & 0&-x &0 & \cdots \\
\cellcolor{lightgray}&\cellcolor{lightgray}&\cellcolor{lightgray}&0&0&0& \cdots\\
\cellcolor{lightgray} \vdots & \cellcolor{lightgray}\vdots &\cellcolor{lightgray} \vdots & \vdots& \vdots &\vdots & \ddots\\
\end{array}\,\right|  +  b_{m-1}(x)
\end{eqnarray*} 
where $b_m(x) = \det(B_m(x))$ for the $(3m+2) \times (3m+2)$ matrix 
\[
B_m(x) = 
\begin{bmatrix}
W & Z & \cdots & 0 \\
X & Y & \cdots & 0 \\
\vdots & \vdots & \ddots & \vdots \\
0 & 0 & \cdots & \overline Y_{3,3}
\end{bmatrix}
\qquad \text{where} \qquad 
W = 
\begin{bmatrix}
0 & 1 & 1 \\
1 & 1-x & 1 \\
0 & 0 & -x
\end{bmatrix}.
\]
Therefore we have
\[
a_m(x) = x^2(1-x) a_{m-1}(x)  + b_{m-1}(x).
% \det B_{m-1}
\]

Now, subtracting row 1 of $B_m(x)$ from row 2 and expanding by cofactors, we obtain the following:

\begin{eqnarray*}
b_m(x) &=& \det B_m(x) \\
&=&  
\begin{vmatrix}
0 &1 & 1 &0 & 0 &0& \cdots \\
1 & 1-x & 1&0 &0& 0 & \cdots \\
0 & 0 & -x &0 &1 &1 & \cdots\\
1 & 1 & 0& -x&0 & 0 &\cdots \\
0 & 0 &0 & 1&1-x &1 & \cdots \\
0&0&0&0&0&-x& \cdots\\
\vdots & \vdots &\vdots & \vdots& \vdots &\vdots & \ddots \\
\end{vmatrix} 
%\end{eqnarray*}
%
%
%\begin{eqnarray*}
%b_m(x) &=&  
= \begin{vmatrix}
0 &1 & 1 &0 & 0 &0& \cdots \\
1 & -x & 0&0 &0& 0 & \cdots \\
0 & 0 & -x &0 &1 &1 & \cdots\\
1 & 1 & 0& -x&0 & 0 &\cdots \\
0 & 0 &0 & 1&1-x &1 & \cdots \\
0&0&0&0&0&-x& \cdots\\
\vdots & \vdots &\vdots & \vdots& \vdots &\vdots & \ddots \\
\end{vmatrix} \\
&=&
(-1)  
\left|\,\begin{array}{ccccccc}
\cellcolor{lightgray} &\cellcolor{gray} & \cellcolor{lightgray} &\cellcolor{lightgray} & \cellcolor{lightgray} &\cellcolor{lightgray}& \cellcolor{lightgray}\cdots \\
1 & \cellcolor{lightgray} & 0&0 &0& 0 & \cdots \\
0 & \cellcolor{lightgray} & -x &0 &1 &1 & \cdots\\
1 & \cellcolor{lightgray} & 0& -x&0 & 0 &\cdots \\
0 & \cellcolor{lightgray} &0 & 1&1-x &1 & \cdots \\
0&\cellcolor{lightgray}&0&0&0&-x& \cdots\\
\vdots & \cellcolor{lightgray}\vdots &\vdots & \vdots& \vdots &\vdots & \ddots \\
\end{array}\,\right|+
(1)  
\left|\,\begin{array}{ccccccc}
\cellcolor{lightgray} &\cellcolor{lightgray} & \cellcolor{gray} &\cellcolor{lightgray} & \cellcolor{lightgray} &\cellcolor{lightgray}& \cellcolor{lightgray} \cdots \\
1 & -x & \cellcolor{lightgray}&0 &0& 0 & \cdots \\
0 & 0 &\cellcolor{lightgray} &0 &1 &1 & \cdots\\
1 & 1 & \cellcolor{lightgray}& -x&0 & 0 &\cdots \\
0 & 0 &\cellcolor{lightgray}& 1&1-x &1 & \cdots \\
0&0&\cellcolor{lightgray}&0&0&-x& \cdots\\
\vdots & \vdots &\cellcolor{lightgray}\vdots & \vdots& \vdots &\vdots & \ddots \\
\end{array}\,\right|  
\end{eqnarray*}
Subtracting the first (non-grayed) row from the third (non-grayed) row in the second matrix above, we obtain that $b_m(x)$ equals
\begin{eqnarray*}
&=&
(-1)(-x)  
\left|\,\begin{array}{ccccccc}
\cellcolor{lightgray} &\cellcolor{gray} & \cellcolor{lightgray} &\cellcolor{lightgray} & \cellcolor{lightgray} &\cellcolor{lightgray}& \cellcolor{lightgray}\cdots \\
1 & \cellcolor{lightgray} & \cellcolor{lightgray}&0 &0& 0 & \cdots \\
\cellcolor{lightgray} & \cellcolor{lightgray} & \cellcolor{gray} &\cellcolor{lightgray} &\cellcolor{lightgray} &\cellcolor{lightgray} & \cellcolor{lightgray} \cdots\\
1 & \cellcolor{lightgray} & \cellcolor{lightgray}& -x&0 & 0 &\cdots \\
0 & \cellcolor{lightgray} &\cellcolor{lightgray} & 1&1-x &1 & \cdots \\
0&\cellcolor{lightgray}&\cellcolor{lightgray}&0&0&-x& \cdots\\
\vdots & \cellcolor{lightgray}\vdots &\cellcolor{lightgray}\vdots & \vdots& \vdots &\vdots & \ddots \\
\end{array}\,\right| 
+ (1)  
\left|\,\begin{array}{ccccccc}
\cellcolor{lightgray} &\cellcolor{lightgray} & \cellcolor{gray} &\cellcolor{lightgray} & \cellcolor{lightgray} &\cellcolor{lightgray}& \cellcolor{lightgray} \cdots \\
1 & -x & \cellcolor{lightgray}&0 &0& 0 & \cdots \\
0 & 0 &\cellcolor{lightgray}  &0 &1 &1 & \cdots\\
0 & 1+x & \cellcolor{lightgray}& -x&0 & 0 &\cdots \\
0 & 0 &\cellcolor{lightgray} & 1&1-x &1 & \cdots \\
0&0&\cellcolor{lightgray}&0&0&-x& \cdots\\
\vdots & \vdots &\cellcolor{lightgray}\vdots & \vdots& \vdots &\vdots & \ddots \\
\end{array}\,\right|\\
&=&
(-1)(-x)(1)  
\left|\,\begin{array}{ccccccc}
\cellcolor{lightgray} &\cellcolor{gray} & \cellcolor{lightgray} &\cellcolor{lightgray} & \cellcolor{lightgray} &\cellcolor{lightgray}& \cellcolor{lightgray}\cdots \\
\cellcolor{gray} & \cellcolor{lightgray} & \cellcolor{lightgray}&\cellcolor{lightgray} &\cellcolor{lightgray}& \cellcolor{lightgray} & \cellcolor{lightgray}\cdots \\
\cellcolor{lightgray} & \cellcolor{lightgray} & \cellcolor{gray} &\cellcolor{lightgray} &\cellcolor{lightgray} &\cellcolor{lightgray} & \cellcolor{lightgray} \cdots\\
\cellcolor{lightgray} & \cellcolor{lightgray} & \cellcolor{lightgray}& -x&0 & 0 &\cdots \\
\cellcolor{lightgray} & \cellcolor{lightgray} &\cellcolor{lightgray} & 1&1-x &1 & \cdots \\
\cellcolor{lightgray}&\cellcolor{lightgray}&\cellcolor{lightgray}&0&0&-x& \cdots\\
\cellcolor{lightgray}\vdots & \cellcolor{lightgray}\vdots &\cellcolor{lightgray}\vdots & \vdots& \vdots &\vdots & \ddots \\
\end{array}\,\right| 
+
(1)(1)  
\left|\,\begin{array}{ccccccc}
\cellcolor{lightgray} &\cellcolor{lightgray} & \cellcolor{gray} &\cellcolor{lightgray} & \cellcolor{lightgray}&\cellcolor{lightgray}& \cellcolor{lightgray} \cdots \\
\cellcolor{gray} & \cellcolor{lightgray} & \cellcolor{lightgray}&\cellcolor{lightgray} &\cellcolor{lightgray}& \cellcolor{lightgray} & \cellcolor{lightgray}\cdots \\
\cellcolor{lightgray} & 0 &\cellcolor{lightgray}  &0 &1 &1 & \cdots\\
\cellcolor{lightgray} & 1+x & \cellcolor{lightgray}& -x&0 & 0 &\cdots \\
\cellcolor{lightgray} & 0 &\cellcolor{lightgray} & 1&1-x &1 & \cdots \\
\cellcolor{lightgray}&0&\cellcolor{lightgray}&0&0&-x& \cdots\\
\cellcolor{lightgray} \vdots & \vdots &\cellcolor{lightgray}\vdots & \vdots& \vdots &\vdots & \ddots \\
\end{array}\,\right| \\
&=&
(-1)(-x)(1)(-x)
\left|\,\begin{array}{ccccccc}
\cellcolor{lightgray} &\cellcolor{gray} & \cellcolor{lightgray} &\cellcolor{lightgray} & \cellcolor{lightgray} &\cellcolor{lightgray}& \cellcolor{lightgray}\cdots \\
\cellcolor{gray} & \cellcolor{lightgray} & \cellcolor{lightgray}&\cellcolor{lightgray} &\cellcolor{lightgray}& \cellcolor{lightgray} & \cellcolor{lightgray}\cdots \\
\cellcolor{lightgray} & \cellcolor{lightgray} & \cellcolor{gray} &\cellcolor{lightgray} &\cellcolor{lightgray} &\cellcolor{lightgray} & \cellcolor{lightgray} \cdots\\
\cellcolor{lightgray} & \cellcolor{lightgray} & \cellcolor{lightgray}& \cellcolor{gray}&\cellcolor{lightgray}& \cellcolor{lightgray} &\cellcolor{lightgray}\cdots \\
\cellcolor{lightgray} & \cellcolor{lightgray} &\cellcolor{lightgray} & \cellcolor{lightgray}&1-x &1 & \cdots \\
\cellcolor{lightgray}&\cellcolor{lightgray}&\cellcolor{lightgray}&\cellcolor{lightgray}&0&-x& \cdots\\
\cellcolor{lightgray}\vdots & \cellcolor{lightgray}\vdots &\cellcolor{lightgray}\vdots &\cellcolor{lightgray} \vdots& \vdots &\vdots & \ddots \\
\end{array}\,\right| 
+
(1)(1)(-1-x) 
\left|\,\begin{array}{ccccccc}
\cellcolor{lightgray} &\cellcolor{lightgray} & \cellcolor{gray} &\cellcolor{lightgray} & \cellcolor{lightgray} &\cellcolor{lightgray}& \cellcolor{lightgray} \cdots \\
\cellcolor{gray} & \cellcolor{lightgray} & \cellcolor{lightgray}&\cellcolor{lightgray} &\cellcolor{lightgray}& \cellcolor{lightgray} & \cellcolor{lightgray}\cdots \\
\cellcolor{lightgray} & \cellcolor{lightgray} &\cellcolor{lightgray} &0 &1 &1 & \cdots\\
\cellcolor{lightgray} & \cellcolor{gray} & \cellcolor{lightgray}& \cellcolor{lightgray}&\cellcolor{lightgray} & \cellcolor{lightgray}&\cellcolor{lightgray}\cdots \\
\cellcolor{lightgray} & \cellcolor{lightgray} &\cellcolor{lightgray} & 1&1-x &1 & \cdots \\
\cellcolor{lightgray}&\cellcolor{lightgray}&\cellcolor{lightgray}&0&0&-x& \cdots\\
\cellcolor{lightgray} \vdots &\cellcolor{lightgray} \vdots &\cellcolor{lightgray}\vdots & \vdots& \vdots &\vdots & \ddots \\
\end{array}\,\right|, \\
 &=& (-1)(-x)(1)(-x) \det(A_{m-1} - xI) + (1)(1)(-1-x) \det B_{m-1}(x)  
\end{eqnarray*}
Thus we have the recursion formulas
\begin{eqnarray*}
a_m(x) &=& (x^2-x^3) a_{m-1}(x) + b_{m-1}(x) \\
b_m(x) &=& -x^2 a_{m-1}(x) - (1+x) b_{m-1}(x)
\end{eqnarray*}
The first equation gives a formula for each term of the $b$-sequence in terms of the $a$-sequence. Substitituting that formula for $b_m$ and $b_{m-1}$ in the second equation gives the desired recurrence for $a_m(x)$. 
\end{proof}

The first few characteristic polynomials $a_m(x)$ are shown in Table \ref{tab:largest_root}. We now give an explicit formula for them. 

\begin{proposition}\label{prop:charpoly}
The polynomial $a_m(x)$ is given explicitly by the formula
%
%\begin{eqnarray*}
%a_0(x) &=& 1-x \\x
%a_1(x) &=&  x^4-2x^3+x^2-1 \\
%a_m(x) &=& (-x^3+x^2-x-1) a_{m-1}(x) - x^4 a_{m-2}(x) \qquad \text{ for } m \geq 2\\
%\end{eqnarray*}
\[
a_m(x) = \alpha_+(x) \beta_+(x)^m + \alpha_-(x) \beta_-(x)^m
\]
for all $m \geq 0$, where 
\begin{eqnarray*}
\alpha_\pm(x) &=& \pm \frac{x^4-2x^3-1}{2 \sqrt{(x^4-1)(x^2-2x-1)}} + \frac{1-x}2, \\
\beta_\pm(x) &=& \frac{-x^3+x^2-x-1 \pm \sqrt{(x^4-1)(x^2-2x-1)}}{2}.  
\end{eqnarray*}
for any 
%$x$ such that $\gamma(x) = x^6-2x^5-x^4-x^2+2x+1$ is not zero; that is, 
$x \notin \{\pm 1, \pm i, 1 \pm \sqrt{2}\}$.

%\textbf{\textsf{TO DO}}: This equation is certainly true algebraically, in the ring of ``formal power series". To prove that this is also true  analytically, we need to choose a domain in $\mathbb{C}$ that does not give singularities. This should not be too hard but it is necessary.
\end{proposition}

\begin{proof}
The relation 
\[
a_m(x) = (-x^3+x^2-x-1) a_{m-1}(x) - x^4 a_{m-2}(x) \qquad \text{ for } m \geq 2
\]
is a linear recurrence with constant coefficients in the field of power series in $x$.
To solve it, we need to compute the characteristic polynomial of the recurrence:
\[
\beta^2 + (x^3-x^2+x+1)\beta + x^4 = 0,
\]
whose discriminant is 
\[
\gamma(x) := (x^3-x^2+x+1)^2 - 4x^4 = x^6-2x^5-x^4-x^2+2x+1 = (x^4-1)(x^2-2x-1);
\]
its roots are $\pm 1, \pm i, 1 \pm \sqrt{2}$.

For $\gamma(x) \neq 0$, this polynomial has two different roots $\beta^{\pm}$, given in the statement of the proposition. The general theory of recurrences tells us that there must exist constants $\alpha_{\pm}$ such that
\[
a_n = \alpha_+ \beta_+^n +  \alpha_- \beta_-^n
\]
for all natural numbers $n$. Substituting $n=0, 1$ gives the system of equations
\begin{eqnarray*}
1-x &=& \alpha_+ + \alpha_- \\
x^4-2x^3+x^2-1  &=&  \alpha_+ \beta_+ +  \alpha_- \beta_-
\end{eqnarray*}
whose solution for $\alpha_+$ and $\alpha_-$ is as given. 
\end{proof}

It feels like a small miracle that the discriminant $\gamma(x)$ arising in this problem has such a nice factorization; it would be interesting to explain this conceptually. 

\begin{proposition}\label{prop:c_m}
The number $c_m(n)$ of monotone paths of length $n$ in the strip of height $m$ is the sum of the entries in row $00$ of $A_m^n$; equivalently, 
\[
c_m(n) = [1 0 \cdots 0] A_m^n [1 1 \cdots 1]^T.
\]
\end{proposition}

\begin{proof} The number $c_m(n)$ equals the number of paths of length $n$ in the transfer graph $G_m$ that start at $00$, so it is the sum of the entries in the first row of $A_m^n$, which is given by $[1 0 \cdots 0] A^n [1 1 \cdots 1]^T.$
\end{proof}

For any fixed height $m$, Proposition \ref{prop:c_m} may be used to give an explicit formula for the function $c_m(n)$ as $n$ varies. To do so, one computes the Jordan normal form $J_m$ of $A_m$, so that if $A_m = KJ_mK^{-1}$ then $A_m^n = KJ_m^nK^{-1}$, where the powers of $J_m$ are easily computed; see for example \cite[(9.1.4)]{Golub13}. 
We get an even nicer answer for the generating function of this sequence.

\begin{theorem}\label{thm:detformula}
Let $m$ be a fixed positive integer. The generating function for the number $c_m(n)$ of monotone paths of length $n$ in the strip of height $m$ is 
\[
\sum_{n \geq 0}  c_m(n)x^n =  \frac{det(I-x A_m \, ; \, 00)}{det(I-x A_m)}
\]
where $(I-xA_m \, ; \, 00)$ is the matrix obtained from $I-xA_m$ by replacing every entry in the first column $00$ with a $1$.
\end{theorem}

\begin{proof} For any $N \times N$ matrix $A$ we have % \cite[Theorem 4.7.2]{Sta11} that
\[
\sum_{n \geq 0} A^n_{ij} x^n = 
\left(\sum_{n \geq 0} (Ax)^n\right)_{ij} = 
\left((I-Ax)^{-1}\right)_{ij} = 
\frac{(-1)^{i+j} \det(I-xA \, : \, j,i)}{\det(I-xA)}
\]
where $(K \, : \, j,i)$ denotes the matrix $K$ with row $j$ and column $i$ removed. Therefore
\begin{eqnarray*}
\sum_{j=1}^N\sum_{n \geq 0} A^n_{ij} x^n &=& \sum_{j=1}^N \frac{(-1)^{j+i} \det(I-xA \, : \, j,i)}{\det(I-xA)}, \\
&=& \frac{det(I-x A \, ; \, j)}{\det(I-x A)},
\end{eqnarray*}
where $(K \, ; \, j)$ denotes the matrix $K$ with every entry in row $j$ replaced by a $1$; in the last step we use the expansion of $\det(I-x A)$ by cofactors along the $j$th row. Applying this formula to Proposition \ref{prop:c_m} gives the desired result.
\end{proof}
%
%\[
%\sum_{n \geq 0} c_m(n)x^n = 1+x\sum_{n \geq 1} [1 0 \cdots 0] A_m^{n-1}x^{n-1} [1 1 \cdots 1]^T.
%\]
%\end{proof}

\begin{remark}
Using Theorem \ref{thm:detformula} for the transition matrix $A_2$ of Figure \ref{fig:graph}, we readily obtain the generating function for monotone paths in a strip of height $2$:
\begin{eqnarray*}
\sum_{n \geq 0} c_2(n)x^n &=& \frac{1-x+x^2+x^3-x^4+x^5+x^6}{1-3x+3x^2-x^3-2x^4+2x^5-x^6-x^7} \\
&=& \frac{1+x^2+x^3}{1-2x+x^2-x^3-x^4} \\
&=& 1+2x+4x^2+8x^3+15x^4+28x^5+53x^6+101x^7+\cdots
\end{eqnarray*}
This matches the computation of Williams in \cite[(3)]{Wil96} and Ardila, Bastidas, Ceballos, and Guo in \cite[Corollary 3.6]{Ard17}; Williams also computed the generating function for $c_3(n)$. Their methods require a careful analysis for each fixed height $m$, and the complexity of that analysis grows with $m$. The advantage of our method is that it works uniformly for any height $m$.
\end{remark}

\subsection{\textsf{Asymptotic analysis}}

Theorem \ref{thm:detformula} shows that for any fixed $n$, the generating function for $c_m(n)$ is a rational function in $x$. This implies that the asymptotic growth of $c_m(n)$ is controlled by the roots of the denominator of that rational function. Let us make that precise.

\begin{theorem} \label{th:rationalfunctions}\cite[Theorem 4.1.1]{Sta11}
Let $\alpha_1, \ldots, \alpha_d$ be a fixed sequence of complex numbers, $d \geq 1$ and $\alpha_d \neq 0$. Let $Q(x) = 1+\alpha_1x + \cdots + \alpha_d x^d = \prod_{i=1}^k (1-\lambda_i x)^{d_i}$ where the $\lambda_i$s are distinct and $d_1+\cdots+d_k=d$. The following conditions on a function $f: \mathbb{N} \rightarrow \mathbb{C}$ are equivalent:
\begin{enumerate}
\item There is a polynomial $P(x)$ of degree less than $d$ and polynomials $p_i(x)$ with $\deg p_i(x) < d_i$ for $i=1, \ldots, k$ such that
\[
\sum_{n \geq 0} f(n)x^n = \frac{P(x)}{Q(x)} = \sum_{i=1}^k \frac{p_i(x)}{(1-\lambda_i x)^{d_i}}.
\]
\item
For all $n \geq 0$,
\[
f(n+d) + \alpha_1 f(n+d-1) + \alpha_2 f(n+d-2) + \cdots + \alpha_d f(n)=0.
\]
\item
There exist polynomials $P_i$ with $\deg P_i(x) < d_i$ for $i=1, \ldots, k$ such that, for all $n \geq 0$,
\[
f(n) = \sum_{i=1}^k P_i(n) \lambda_i^n.
\]
\end{enumerate}
\end{theorem}

If one is interested in the asymptotic growth of $f(n)$, one needs to pay attention to the dominant terms of the expression in Theorem \ref{th:rationalfunctions}.3. This can be a subtle matter, because our polynomial $Q(x)$ 
%matrix can have complex eigenvalues, 
can have complex roots $1/\lambda_i$, which can occur with multiplicites. 
If we are in the fortunate situation where a simple real root dominates the others, the situation is simpler, as follows. 

Say two sequences $(f(n))_{n \geq 0}$ and $(g(n))_{n \geq 0}$ are \emph{asymptotically equivalent}, and  write
\[
f(n) \sim g(n) \qquad \text{ if } \qquad \lim_{n \rightarrow \infty} \frac{f(n)}{g(n)} = 1
\]

%However, we are in the fortunate situation where a single real eigenvalue dominates the others, as we now explain.

\begin{lemma} \label{lem:asymp}
Assume that the conditions of Theorem \ref{th:rationalfunctions} hold, and that $\lambda=\lambda_1 \in \mathbb{R}_{>0}$ has the property that $d_1=1$ and $|\lambda_i| < \lambda$ for $i=2, \ldots, k$. Then
 \[
f(n) \sim  a \cdot \lambda^n \qquad \text{ where } \qquad a = \frac{-\lambda P(1/\lambda)}{Q'(1/\lambda)}.
\]
\end{lemma}

\begin{proof}
Since $d_1=1$ the corresponding polynomial $p_1(x)=a$ is a constant, and the Taylor expansion of
\[
\sum_{n \geq 0} f(n)x^n = \frac{P(x)}{Q(x)} = \frac{a}{1-\lambda x} +  \sum_{i=2}^k \frac{p_i(x)}{(1-\lambda_i x)^{d_i}}
\]
gives $f(n) \sim a \cdot \lambda^n$. To identify the constant $a$, notice that
\begin{eqnarray*}
a &=& \lim_{x \rightarrow 1/\lambda} \, \frac{P(x)}{Q(x)} (1-x\lambda) \\
&=& \lim_{x \rightarrow 1/\lambda} \, \frac{-\lambda P(x) + P'(x)(1-x\lambda)}{Q'(x)} \\
&=& \frac{-\lambda P(1/\lambda)  }{Q'(1/\lambda)} 
\end{eqnarray*}
by L'H\^opital's rule. \end{proof}

Now let us apply this framework to the monotone paths that interest us.
Let $\lambda_1, \ldots, \lambda_k$ be the eigenvalues of our transition matrix $A_m$ with respective multiplicities $d_1, \ldots, d_k$, so that 
\[
\det(I-A_mx) = \prod_{i=1}^k (1-\lambda_ix)^{d_i}.
\]
Theorem \ref{thm:detformula} tells us that Theorem \ref{th:rationalfunctions} applies to the sequence $c_m(n)$. We can use it to immediately read off a linear recurrence relation for $c_m(n)$, as well as an explicit formula:
\begin{equation}\label{eq:c_m}
c_m(n) = \sum_{i=1}^k P_i(n) \lambda_i^n
\end{equation}
%where $P_i$ is a polynomial of degree less than $d_i$. 
For a fixed height $m$, the polynomials $P_i$ --  and hence the exact formula for $c_m(n)$ -- can be computed explicitly: this is done by writing down the partial fraction decomposition of the right hand side of Theorem \ref{thm:detformula}, and then computing its Taylor series.

%If one is only interested in the asymptotic growth of $c_m(n)$, one needs to pay attention to the dominant terms of \eqref{eq:c_m}. In principle this could be a subtle matter, because our matrix can have complex eigenvalues, which can occur with multiplicites. However, 
Additionally, we are in the fortunate situation where a single real eigenvalue of $A_m$ dominates the others and Lemma \ref{lem:asymp} applies, as we now explain.

\begin{definition}
A square matrix $A$ is \emph{primitive} if it is entrywise non-negative and some positive power $A^k$ is entrywise positive. 
\end{definition}

\begin{theorem}[Perron-Frobenius]
Every primitive matrix $A$ has a \emph{Perron eigenvalue} $r$: this is a positive eigenvalue $r > 0$ such that $r > |\lambda|$ for any other eigenvalue $\lambda$ of $A$. Furthermore, $r$ is a simple eigenvalue of $A$ -- that is, it has multiplicity 1 -- and it is between the minimum and the maximum column sums of $A$.
\end{theorem}

\begin{lemma} \label{lem:Perron}
The matrix $A_m$ is primitive and its Perron eigenvalue $r_m=\lambda(A_m)$ satisfies $1 \le r_m \le 3$.
%has a \emph{Perron eigenvalue} $0 < r_m \leq 3$ : this is a real eigenvalue $r_m > 0$ such that $r_m > |\mu|$ for any other eigenvalue $\mu$ of $A_m$. Furthermore, $r_m$ is a simple eigenvalue of $A_m$.
\end{lemma}

\begin{proof}
%The Perron-Frobenius Theorem establishes these facts for any non-negative matrix $A$ such that $A^k$ if positive for some positive integer $k$. 
Our matrix $A_m$ is clearly non-negative. 
By inspecting the graph $G_m$, we see that for any vertices $u$ and $v$ there is a walk from $u$ to $v$ of length at most $2m+1$. This walk must use at least one vertex $(i,i)$, and adding loops $(i,i) \rightarrow (i,i)$ to the walk, one can extend it to have length exactly $2m+1$. Therefore $A_m^{2m+1}$ is positive. The first claim follows, and the second one follows readily from the Perron-Frobenius theorem and the observation that the indegree and outdegree of any vertex of $G_m$ is at least $1$ and at most $3$.
\end{proof}

\begin{corollary}\label{cor:c_mPerron}
The Perron eigenvalue If $r_m:= \lambda(A_m)$ of the matrix $A_m$, that is, the largest positive root of the polynomial $a_m(x)$ of Proposition \ref{prop:charpoly},  is the exponential growth rate constant for monotone paths in a strip of height $m$; more precisely,
\[
c_m(n) \sim q_m \cdot r_m^n
\]
for a constant $q_m$.
% $a_m = \frac{-r_m P_m(1/r_m)}{Q_m'(1/r_m)}$ for $P_m(x) = \det(I-xA_m ; 00)$ and $Q_m(x) = \det(I-xA_m)$.
\end{corollary}

\begin{proof}
This follows readily from Lemma \ref{lem:asymp} and Lemma \ref{lem:Perron}.% since the Perron eigenvalue $\lambda=\lambda(A_m)$ has multiplicity $1$, so the polynomial corresponding to it in \eqref{eq:c_m} is a constant. For every other eigenvalue we have $|P_i(n) \lambda_i^n|/a_n r_m^n \rightarrow 0$.
\end{proof}

\begin{table}
\centering
\resizebox{\columnwidth}{!}{%
\begin{tabular}{|c|c|c|}
\hline
$m$ & $a_m(x) = \det(A_m - xI)$ & $r_m$\\
\hline
$0$ & $- x+ 1$ & $1$\\
$1$ & $x^4-2x^3+x^2-1$ & $ 1.6180\ldots $\\
$2$ & $-x^7 + 3 x^6 - 3 x^5 + x^4 + 2 x^3 - 2 x^2 + x + 1$ & $ 1.8971\ldots$\\
$3$ & $x^{10} - 4 x^9 + 6 x^8 - 4 x^7 - 2 x^6 + 6 x^5 - 5 x^4 + 2 x^2 - 2 x - 1$
 & $2.0507\ldots$\\
$4$ & $-x^{13} + 5 x^{12} - 10 x^{11} + 10 x^{10} - x^9 - 11 x^8 + 15 x^7 - 7 x^6 - 4 x^5 + 8 x^4 - 3 x^3 - x^2 + 3 x + 1$ & $2.1444\ldots$\\
$\downarrow$ & & $\downarrow$\\
$\infty$ & & $1 + \sqrt 2$ \\
\hline
\end{tabular}%
}
\caption{Characteristic polynomial $a_m(x)$ and Perron eigenvalue $r_m$ of the transition matrix $A_m$. The number $r_m$ is the exponential growth constant of the number $c_m(n)$ of monotone paths of length $n$ in a tunnel of height $m$.}
\label{tab:largest_root}
\end{table}

Table \ref{tab:largest_root} shows the growth constants $r_m$ for monotone paths in height $m$; that is, the Perron eigenvalues of the transition matrices $A_m$, for the first few values of $m$. 
We note that our description of the growth constant $r_m$ as the largest real root of the polynomial $a_m(x)$ in Proposition \ref{prop:charpoly} is the most explicit possible, because Galois theory tells us that there is no exact formula for it. For example, $a_3(x)$ factors into two irreducible  quintics, and the quintic $x^5-2x^4+x^2-2x-1$ that has $r_3$ as a root has full Galois group $S_5$. 

We now offer an optimal upper bound for the growth constants $r_m$ as the height $m$ of the tunnel grows.

%The growth of $c_m(n)$ is determined by the Perron eigenvalue of $A_m$.
%
%%
%%\begin{corollary}\label{cor:lim_ratio_cm}
%%\[
%%\lim_{n \to \infty} \frac{c_{m+1}(n)}{c_m(n)} = \lim_{n \to \infty} \frac{a_{m+1}\lambda_{m+1}^n}{a_{m}\lambda_{m}^n}.
%%\]
%%\end{corollary}

\begin{proposition}\label{prop:Perron}
The Perron eigenvalues $r_0, r_1, r_2, \ldots$ of the matrices $A_0, A_1, A_2, \ldots$ satisfy
\[
1=r_0 < r_1 < r_2 < \cdots, \mathrm{ \ and \ } \lim_{m \rightarrow \infty} r_m = 1+\sqrt{2}.
\]
%\approx 2.43477\ldots$ is the positive root of the polynomial $x^6-2x^5-x^4-x^2+2x+1=0$.
\end{proposition}

\begin{proof}
If we had $r_{m+1} < r_m$ for some $m$,  Corollary \ref{cor:c_mPerron} would imply
\[
\lim_{n \to \infty} \frac{c_{m+1}(n)}{c_m(n)} = \lim_{n \to \infty} \frac{a_{m+1}}{a_{m}} \left(\frac{r_{m+1}}{r_{m}}\right)^n = 0,
\]
contradicting the fact that $c_m(n) \leq c_{m+1}(n)$. Also, if we had $r_{m+1}=r_m$, then this would be a common root of the polynomials $a_{m+1}(x)$ and $a_m(x)$, and the recurrence of Lemma \ref{lem:rec} would imply that it is also a root of $a_{m-1}(x), a_{m-2}(x), \ldots, a_0(x)$. However $a_0$ and $a_1$ don't have a root in common. This proves the first claim. 

%because every walk of length $n$ that fits in a strip of height $m$ also fits in a strip of height $m+1$.  
% To prove that the $r_i$s are increasing, let us suppose for the sake of contradiction that $r_{m+1} < r_m$ for some $m \ge 0$. 
%Every walk of length $n$ that fits in a strip of height $m$ also fits in a strip of height $m+1$, so $c_m(n) \leq c_{m+1}(n)$. But Corollary \ref{cor:c_mPerron} tells us that
%\[
%\lim_{n \to \infty} \frac{c_{m+1}(n)}{c_m(n)} = \lim_{n \to \infty} \frac{a_{m+1}}{a_{m}} \left(\frac{r_{m+1}}{r_{m}}\right)^n = 0,
%\]
%a contradiction.  This completes the proof of the first part.

For the second part, notice that the Perron eigenvalues are increasing and bounded above by $3$ by Lemma \ref{lem:Perron}, so the sequence does converge. Let the limit be
\[
\lim_{m \rightarrow \infty} r_m = r
\]
and assume, for the sake of contradiction, that $r \neq 1 + \sqrt 2$.

Since $r_m$ is an eigenvalue for $A_m$, it is a root of $a_m(x)$, so Proposition \ref{prop:charpoly} gives
\[
0 = a_m(r_m) =  \alpha_+(r_m) \beta_+(r_m)^m + \alpha_-(r_m) \beta_-(r_m)^m.
\]
Let us write $\alpha(x) = \alpha_+(x)/\alpha_-(x)$ and $\beta(x) = \beta_-(x)/\beta_+(x)$; these are well defined and non-zero for all but a finite number of values $x$. Thus for all sufficiently large $m$ we have
%\begin{eqnarray*}
%-\alpha(r_m) &=& \beta(r_m)^m \\
%\log(-\alpha(r_m)) &=& m \log \beta(r_m) 
%\end{eqnarray*}
\[
-\alpha(r_m) = \beta(r_m)^m 
\]
and
\begin{equation}\label{eq:limit}
\frac1m \log(-\alpha(r_m)) =  \log \beta(r_m). 
\end{equation}

We now wish to take limits, but since the discriminant $\gamma(x) = (x^4-1)(x^2-2x-1)$ -- whose square root arises in $\alpha_\pm$ and $\beta_\pm$ -- can be negative, we need to regard these as complex functions. Making a branch cut along the ray spanned by $1+i$  gives rise to two branches of the square root function $\pm \sqrt{z}$, each of which is continuous in the domain
\[
D = \mathbb{C} \setminus (1+i) \, \mathbb{R}_{\geq 0}.
\]
Now $r_m \rightarrow r$ implies 
$\gamma(r_m) \rightarrow \gamma(r)$. This limit is real and non-zero, since $r>1$ and we assumed $r \neq 1+\sqrt{2}$. Thus $\sqrt z$ is continuous at $\gamma(r)$, 
 so $\sqrt{\gamma(r_m)} \rightarrow \sqrt{\gamma(r)}$. This implies that $\alpha(r_m) \rightarrow \alpha(r)$ and $\beta(r_m) \rightarrow \beta(r)$.

One may verify computationally that $\alpha_+$, $\alpha_-$, $\beta_+$, $\beta_-$ have no real roots. Furthermore, $\alpha_-$ and $\alpha_+$ only have two positive poles, located at 
 $x=1$ and $x = 1+\sqrt{2}$, the positive roots of the polynomial $\gamma(x)$. Thus we can choose a branch of the logarithm function that is continuous at $-\alpha(r)$ and $\beta(r)$.
  Since $\alpha(r) \neq 0$, the left hand side of \eqref{eq:limit} converges to $0$, while the right hand side converges to $\log \beta(r)$; this means that $\beta(r)=1$.
Thus $\beta_-(r) = \beta_+(r)$, which implies that $\gamma(r)=0$, a contradiction.
 
 We conclude that indeed $r = 1+\sqrt 2$ as desired.
 %$x^6-2x^5-x^4-x^2+2x+1=(x^4-1)(x^2-2x-1) $. If $L \neq r$ then we have
%% the positive root of $x^6-2x^5-x^4-x^2+2x-1=0$. If $r = \lim_{m \rightarrow \infty} \lambda_m$ does not equal this pole, then we have
%%\[
%%\frac{-\alpha_-(\lambda_m)}{\alpha_+(\lambda_m)} =
%%\left(\frac{\beta_+(\lambda_m)}{\beta_-(\lambda_m)}\right)^m
%%\]
%%so
%\[
%-\frac{\alpha_-(L)}{\alpha_+(L)} = 
%\left(
%\lim_{m \rightarrow \infty}
%\frac{\beta_+(L)}{\beta_-(L)}\right)^m.
%\]
%The right hand side can only be $0, 1,$ or $\infty$, and the left hand side cannot equal $0$ or $\infty$, so we must have 
%$\beta_+(L) = \beta_-(L)$; that is, 
%$L^6-2L^5-L^4-L^2+2L+1 =(L^4-1)(L^2-2L-1)=0$. The only positive real solutions to this equation are $1$ and  $1+\sqrt{2}$ Since $L > \lambda_1 \approx 1.61$, we must have $L= 1+\sqrt{2}$. The desired result follows.
\end{proof}

%\textbf{\textsf{TO DO}}: There is a small issue in the proof above: we have $\lambda_m <r$ for all $m$, and this means that $\beta_+(\lambda_m)$ and $\beta_-(\lambda_m)$ are imaginary! So I guess these limits have to be taken in the complex plane, and this means we need to analyze $\alpha_\pm(x)$ and $\beta_\pm(x)$ as functions on the complex plane. This should be fine, but will require a bit of complex analysis -- branches of the square root function, things like that. Mariana Smit told me that as long as we do computations away from the line $\mathbb{R}_{<0}$ -- or away from any ray coming out of the origin, we can choose a branch of the square root function that makes this work.

We note that these results are consistent with the observation, recorded by Janse van Rensburg, Prellberg, and Rechnitzer in \cite[Lemma 2.1]{JPR}, that the growth constant for the monotone paths in the first quadrant.
equals $1 + \sqrt{2}$. 
Remarkably, they showed that the growth constant 
still equals $1 + \sqrt{2}$ when considering monotone paths in the wedges bound by lines $y=0$ and $y=mx$, or bound by lines $y=-mx$ and $y=mx$, for any integer slope $m$.
It would be interesting to generalize our results to those settings.

\section{\textsf{Using CAT(0) cube complexes to find a small bottleneck}}\label{sec:CAT(0)}

For many Markov chains $M$, the transition kernel is the skeleton of a \emph{CAT(0) cube complex}. When that is the case, Ardila, Owen, and Sullivant \cite{AOS} showed how to associate a \emph{poset with inconsistent pairs (PIP)} $P_M$ to $M$.
The central idea, which bears repeating, that gave rise to this paper is the following:

\begin{idea} \label{idea}
When the transition kernel of a Markov chain $M$ is a CAT(0) cube complex, one can use the corresponding poset with inconsistent pairs (PIP) $P_M$ to find bottlenecks (vertex separators) in the kernel, and obtain upper bounds on the mixing time of $M$.
\end{idea}

In this section we make this statement precise, and in the next section we will use it to bound the mixing time of the Markov chain $M_{m,n}$.

\subsection{\textsf{The cube complex of monotone paths in a strip}}

There are numerous contexts where a discrete system moves according to local, reversible moves. Abrams, Ghrist, and Peterson introduced the formalism of \emph{reconfigurable systems} to model a very wide variety of such contexts. In particular, they showed how a reconfigurable system $X$ leads to a \emph{cube complex} $\S(X)$. Ardila, Bastidas, Ceballos, and Guo described the complex $S_{m,n}$ of  monotone paths in a strip, using the language of robotic arms in a tunnel. We now give their description, and refer the reader to \cite{Abr04, ABY, GhristPeterson} for the general framework.

\begin{definition}
Let $S_{m,n}$ be the \emph{transition kernel} of the Markov chain $M_{m,n}$. Its vertices correspond to the $c_m(n)$ monotone paths of length $n$ in a strip of height $m$. Two vertices are connected to each other if the corresponding paths can be obtained from one another by one of the following moves:

\noindent $\bullet$ 
switch corners: two consecutive steps that go in different directions exchange directions, 

\noindent $\bullet$ 
flip the end: the last step of the path rotates $90^\circ$.
\end{definition}

\begin{figure}
\begin{center}
\includegraphics[width=5.5in]{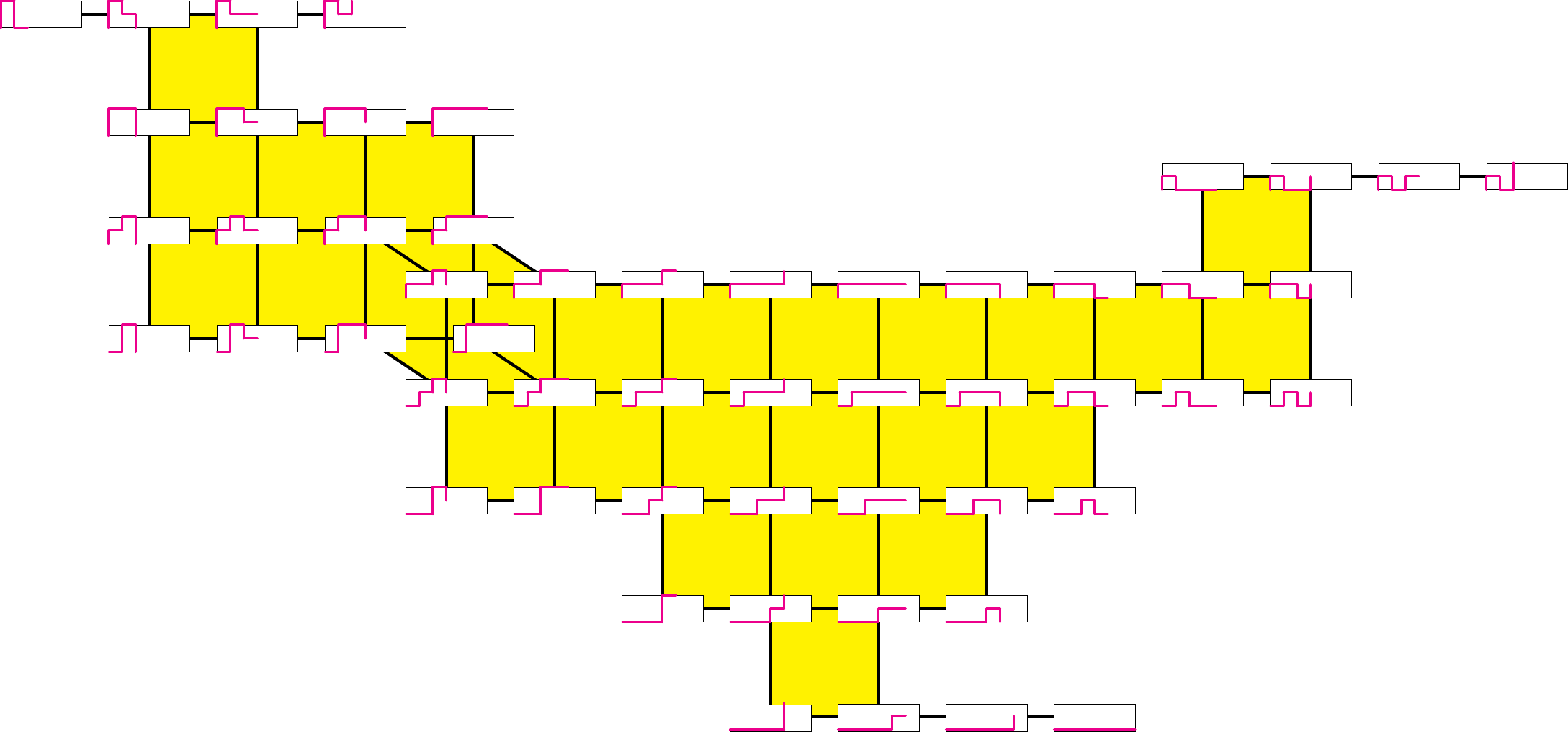}
\end{center}
\caption{The transition kernel $G_{2,6}$ of length $6$ monotone paths in a strip of height $2$.
\label{fig:transition}}
\end{figure}

These moves are illustrated in Figure \ref{fig:switch}. The transition kernel $G_{2,6}$ is shown in Figure \ref{fig:transition}. It feels natural, and is very useful, to ``fill in the cubes" in this graph; let us make this precise.

\begin{definition} 
Given a path $P$ and two moves available at $P$, say these two moves are \emph{compatible} if no step of the path is involved in both of them.
\end{definition}

Intuitively, two moves $M_1$ and $M_2$ on a path $P$ are compatible when they are ``physically independent" from each other in $P$. Performing $M_1$ and then $M_2$ gives the same result as performing $M_2$ and then $M_1$, so we can imagine that $M_1$ and $M_2$ can be performed simultaneously if desired.

For a path $P$ and $k$ moves of $P$ that are pairwise compatible, we can obtain $2^k$ different paths by performing any subset of these $k$ moves to $P$. These $2^k$ vertices form the graph of a $k$-dimensional cube in $\mathcal{S}_{m,n}$. 

\begin{definition}
Let $\mathcal{S}_{m,n}$ be the \emph{transition cube complex} of the Markov chain $M_{m,n}$. Its vertices correspond to the $c_m(n)$ monotone paths of length $n$ in a strip of height $m$. Its $k$-dimensional cubes correspond to the $k$-tuples of pairwise compatible moves.
\end{definition}

This cube complex is naturally a metric space, where each individual cube is a unit cube with the standard Euclidean metric.

\subsection{\textsf{CAT(0) cube complexes and posets with inconsistent pairs} }

The cube complex $\S(X)$ of a reconfigurable system is always locally non-positively curved \cite{Abr04, ABY, GhristPeterson} and sometimes globally non-positively curved, or \emph{CAT(0)}. The reader may consult the definitions in the references above. 
When a cube complex is  CAT(0),  Ardila, Owen, and Sullivant \cite{AOS} showed how to find geodesic between any two points under various metrics. This has consequences for robotic motion planning, among others \cite{ABCG, ArdilaNotices}. As we stated in Idea \ref{idea}, this also has implications for the mixing times of Markov chains.

\subsubsection{\textsf{CAT(0) cube complexes: how to define them}}

The CAT(0) property is the metric property of being non-positively curved, as witnessed by the fact that triangles are ``thinner" than in Euclidean space. Let us make this precise for completeness, although we will not use this definition in what follows.

\begin{definition}
Let $X$ be a metric space where there is a unique geodesic (shortest) path between any two points. Consider a triangle $T$ in $X$ of side lengths $a, b, c$, and build a comparison triangle $T_0$ with the same side lengths in Euclidean plane. Consider a chord of length $d$ in $T$ that connects two points on the boundary of $T$; there is a corresponding comparison chord in $T_0$, say of length $d_0$. If for every triangle $T$ in $X$ and every chord in $T$ we have $d \leq d_0$, we say that X is  CAT(0).  
\end{definition}

\subsubsection{\textsf{CAT(0) cube complexes: how to recognize them topologically} }

Testing whether a general metric space is CAT(0) is quite subtle. However, Gromov \cite{Gromov} proved that this is easier to do this if the space is a cubical complex. In a cubical complex, the link of any vertex is a simplicial complex. We say that a simplicial complex $\Delta$ is \emph{flag} if it has no empty simplices; that is, any $d+1$ vertices which are pairwise connected by edges of $\Delta$ form a $d$-simplex in $\Delta$.

\begin{theorem} (Gromov, \cite{Gromov}) 
A cubical complex is CAT(0) if and only if it is simply connected and the link of any vertex is a flag simplicial complex.
\end{theorem}

If one has a reasonably small cubical complex, one can easily use this criterion to determine whether it is  CAT(0).   Roughly speaking, the first  property says that the space should be connected and have no holes. The second one says that if we stand at a vertex and see that our complex contains all the $2$-faces of a $d$-cube that contain $v$, then in fact it also contains that $d$-cube. For example we see, by inspection, that the cube complex of monotone paths $\S_{2,6}$ in Figure \ref{fig:transition} is  CAT(0).  More generally $\S_{m,n}$ is always CAT(0); see Theorem \ref{thm:ABCG}.

\subsubsection{\textsf{CAT(0) cube complexes: how to describe and build them combinatorially}  \label{sec:CAT(0)-PIP}}

Most relevantly to us, Ardila, Owen, and Sullivant \cite{AOS} gave a combinatorial criterion to determine whether a cube complex is CAT(0). They showed that rooted CAT(0) cube complexes are in bijection with \emph{posets with inconsistent pairs} (PIPs). Thus, if we wish to prove that a cube complex is CAT(0), it is sufficient to choose a root for it, and identify the corresponding PIP. Let us describe this carefully now.

\begin{definition}
A \emph{poset with inconsistent pairs (PIP)} is a finite poset $P$, together with a collection of \emph{inconsistent pairs} $\{p,q\}$ -- denoted $p \nleftrightarrow q$ --  such that:
\begin{enumerate}
\item
If $p$ and $q$ are inconsistent, then there is no $r$ such that $r \geq p$ and $r \geq q$.
\item
If $p$ and $q$ are inconsistent and $p' \geq p$ and $q' \geq q$, then $p'$ and $q'$ are inconsistent.
\end{enumerate}
\end{definition}

The \emph{Hasse diagram} of a PIP is obtained by drawing the poset, and connecting each minimal inconsistent pair with a dotted line. An inconsistent pair $p \nleftrightarrow q$ is \emph{minimal} if there is no other inconsistent pair $p' \nleftrightarrow q'$ with $p' \leq p$ and $q' \leq q$. The left panel of Figure \ref{fig:cubical} shows a PIP.

Recall that $I \subseteq P$ is an \emph{order ideal} or \emph{downset} of poset $P$  if $a \leq b$ and $b \in I$ imply $a \in I$. %, and $A \subseteq P$ is an \emph{antichain} if it does not contain two comparable elements. Define
A \emph{consistent downset} 
%and \emph{compatible antichains} to be those 
is one which contains no inconsistent pairs. 

\begin{definition}\label{def:cubePIP}
Let $P$ be a poset with inconsistent pairs. The \emph{rooted cube complex of $P$}, denoted $\S(P)$, is defined as follows:

$\bullet$ vertices:
The vertices of $\S(P)$ are identified with the consistent order ideals of $P$. 

$\bullet$ edges:
There is an edge joining two vertices if the corresponding order ideals differ by a single element. 

$\bullet$ cubes:
More generally, %, and its maximal cubes correspond to the maximal consistent antichains of $P$. 
there is a cube $C(I,L)$ for each pair $(I, L)$ of a consistent order ideal $I$ and a subset $L \subseteq I_{max}$, where $I_{max}$ is the set of maximal elements of $I$. This cube has dimension $|L|$, and its vertices are obtained by removing from $I$ the $2^{|L|}$ possible subsets of $L$. 

The cubes are naturally glued along their faces according to their labels. The root of $\S(P)$ is the vertex corresponding to the empty order ideal.
\end{definition}

We denote the corresponding graph $S(P)$. 
The right panel of Figure \ref{fig:cubical} shows the rooted cube complex $\S(P)$ (rooted at $\emptyset$) corresponding to the PIP on the left panel.

\begin{theorem}[Ardila, Owen, Sullivant]\label{th:poset} \cite{AOS} 
The map $P \mapsto \S(P)$ is a bijection between finite posets with inconsistent pairs and finite rooted CAT(0) cube complexes.
\end{theorem}

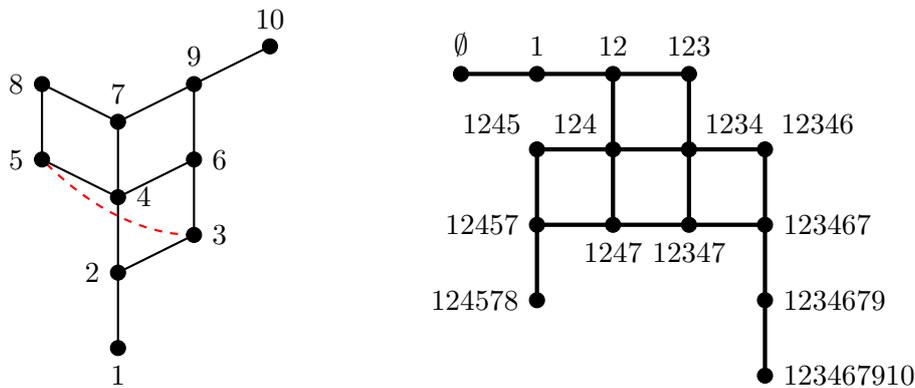
\begin{figure}[h]
\begin{center}
\begin{tikzpicture}
\draw[step=.5cm,white] (-3,-4) grid (0,0);
\draw[thick, -] (0,0) -- (-1,-.5) -- (-2,-1) -- (-3,-.5) -- (-3,-1.5) -- (-2,-2) -- (-2,-1);
\draw[thick, -] (-2,-2) -- (-1,-1.5) -- (-1,-.5);
\draw[thick, -] (-2,-2) -- (-2,-3) -- (-1,-2.5) -- (-1,-1.5);
\draw[thick, -] (-2,-3) -- (-2,-4);
\draw[thick, dashed, red, -] (-1, -2.5) parabola (-3,-1.5);
\node [label=below:1,draw,fill=black,circle,inner sep=2pt,minimum size=2pt] at (-2,-4) {};
\node [label=left:2,draw,fill=black,circle,inner sep=2pt,minimum size=2pt] at (-2,-3) {};
\node [label=right:3,draw,fill=black,circle,inner sep=2pt,minimum size=2pt] at (-1,-2.5) {};
\node [label=right:4,draw,fill=black,circle,inner sep=2pt,minimum size=2pt] at (-2,-2) {};
\node [label=left:5,draw,fill=black,circle,inner sep=2pt,minimum size=2pt] at (-3,-1.5) {};
\node [label=right:6,draw,fill=black,circle,inner sep=2pt,minimum size=2pt] at (-1,-1.5) {};
\node [label=left:8,draw,fill=black,circle,inner sep=2pt,minimum size=2pt] at (-3,-.5) {};
\node [label=above:7,draw,fill=black,circle,inner sep=2pt,minimum size=2pt] at (-2,-1) {};
\node [label=above:9,draw,fill=black,circle,inner sep=2pt,minimum size=2pt] at (-1,-.5) {};
\node [label=above:10,draw,fill=black,circle,inner sep=2pt,minimum size=2pt] at (0,0) {};
\end{tikzpicture}
\qquad 
\qquad
\begin{tikzpicture}
\draw[step=1cm,white, thin] (-4,-4) grid (0,0);
\draw[ultra thick, -] (-4,0) -- (-3,0) -- (-2,0) -- (-1,0) -- (-1,-1) -- (0,-1) -- (0,-2) -- (0,-3) -- (0,-4);
\draw[ultra thick, -] (-2,0) -- (-2,-1) -- (-3,-1) -- (-3,-2) -- (-3,-3);
\draw[ultra thick, -] (-2,-1) -- (-1,-1);
\draw[ultra thick, -] (-3,-2) -- (-2,-2) -- (-1,-2) -- (0,-2);
\draw[ultra thick, -] (-2,-1) -- (-2,-2);
\draw[ultra thick, -] ((-1,-1) -- (-1,-2);
\node [label=above:$\emptyset$,draw,fill=black,circle,inner sep=2pt,minimum size=2pt] at (-4,0) {};
\node [label=above:1,draw,fill=black,circle,inner sep=2pt,minimum size=2pt] at (-3,0) {};
\node [label=above:12,draw,fill=black,circle,inner sep=2pt,minimum size=2pt] at (-2,0) {};
\node [label=above:123,draw,fill=black,circle,inner sep=2pt,minimum size=2pt] at (-1,0) {};
\node [label=above left:1245,draw,fill=black,circle,inner sep=2pt,minimum size=2pt] at (-3,-1) {};
\node [label=above left:124,draw,fill=black,circle,inner sep=2pt,minimum size=2pt] at (-2,-1) {};
\node [label=above right:1234,draw,fill=black,circle,inner sep=2pt,minimum size=2pt] at (-1,-1) {};
\node [label=above right:12346,draw,fill=black,circle,inner sep=2pt,minimum size=2pt] at (0,-1) {};
\node [label=left:12457,draw,fill=black,circle,inner sep=2pt,minimum size=2pt] at (-3,-2) {};
\node [label=left:124578,draw,fill=black,circle,inner sep=2pt,minimum size=2pt] at (-3,-3) {};
\node [label=below:1247,draw,fill=black,circle,inner sep=2pt,minimum size=2pt] at (-2,-2) {};
\node [label=below:12347,draw,fill=black,circle,inner sep=2pt,minimum size=2pt] at (-1,-2) {};
\node [label=right:123467,draw,fill=black,circle,inner sep=2pt,minimum size=2pt] at (0,-2) {};
\node [label=right:1234679,draw,fill=black,circle,inner sep=2pt,minimum size=2pt] at (0,-3) {};
\node [label=right:123467910,draw,fill=black,circle,inner sep=2pt,minimum size=2pt] at (0,-4) {};
\end{tikzpicture}
\end{center}
\caption{A PIP $P \cong C_{2,4}$ and the corresponding CAT(0) cubical complex $\S(P) \cong \S_{2,4}$.}
\label{fig:cubical}
\end{figure}

When a cube complex is CAT(0), Ardila, Owen, and Sullivant \cite{AOS} showed how to find geodesic paths between any two points under various metrics. 
This has consequences for robotic motion planning, among others \cite{ABCG, ArdilaNotices}. We will see here that it also has implications for the mixing times of Markov chains.

\subsection{\textsf{The bottleneck lemma for CAT(0) cube complexes}}

%Given a path-connected space $S$, we say that a subspace $T \subset S$ is a \emph{bottleneck} if the complement $S-T$ can be partitioned into two subspaces $L$ and $R$ such that any path from $L$ to $R$ has to pass through $T$.

\begin{definition}
Let $G$ be a connected graph, we say that a set of vertices $T \subset V(G)$ is a \emph{vertex separator} or \emph{bottleneck} if the removal of $T$ and the edges incident to $T$ disconnects the graph. 
\end{definition}

\begin{lemma}\label{lem:bottleneck}
Let $P$ be a poset with inconsistent pairs and $S(P)$ be the graph of the corresponding CAT(0) cubical complex. Let $p \nleftrightarrow q$ be an inconsistent pair of $P$, and 
\begin{eqnarray*}
S(P)_{p \nleftrightarrow q} &=& \text{ vertices of $S(P)$ whose consistent order ideals  contain neither } p \text{ nor } q. \\
S(P)_p &=& \text{ vertices of $S(P)$ whose consistent order ideals  contain } p \\
S(P)_q &=& \text{ vertices of $S(P)$ whose consistent order ideals  contain } q 
\end{eqnarray*}
Then $S(P)_{p \nleftrightarrow q}$ is a bottleneck for $S(P)$ that separates the sets $S(P)_p$ and $S(P)_q$ from each other.
\end{lemma}

\begin{proof} 
Because $p \nleftrightarrow q$ form an inconsistent pair of $P$, every vertex of $S(P)$ lies in exactly one of these three sets. Thus it suffices to show that there cannot be an edge of $S(P)$ connecting a vertex $v \in S(P)_p$ to vertex $w \in S(P)_q$. But vertex $v$ corresponds to an order ideal $I$ containing $p$ (and hence not containing $q$) and vertex $w$ corresponds to an order ideal $J$ containing $q$ (and hence not containing $p$). It follows that the ideals $I$ and $J$ differ by at least two elements, so there cannot be an edge between $v$ and $w$, as desired.
%
%Consider any path from vertex $v \in S(P)_p$ to vertex $w \in S(P)_q$; say that it consists of vertices $v=v_1, v_2, \ldots, v_k=w$ in order. This path corresponds to a sequence of consistent order ideals $I=I_1, I_2, \ldots, I_k=J$ of $P$. Notice that $p \in I$ (and hence $q \notin I$) and $q \in J$ (and hence $p \notin J$)
%
%Since $p \in I=I_1$ and $p \notin J = I_k$, there exists an index $2 \leq i \leq k$ such that $p \in I_{i-1}$ and $p \notin I_i$. Now, since vertices $v_{i-1}$ and $v_i$ are adjacent, ideals $I_{i-1}$ and $I_i$ differ by a single element, which must be $p$. Therefore $I_i = I_{i-1} - p$. But $p \in I_{i-1}$ implies $q \notin I_{i-1}$, and therefore $q \notin I_{i-1}-p=I_i$. It follows that $p,q \notin I_i$, so $v_i \in S(P)_{p \nleftrightarrow q}$.
%
%We conclude that any path from $S(P)_p$ to $S(P)_q$ has to pass through $S(P)_{p \nleftrightarrow q}$, as desired.
\end{proof}

\subsection{\textsf{The cube complex $\S_{m,n}$ of monotone paths in a strip is CAT(0)}}

Recall that $\S_{m,n}$ is the cube complex of monotone paths of length $n$ in a strip of height $m$. For example, for $m=2$ and $n=4$, Figure \ref{fig:configs} shows the 10 possible paths, labeled to match Figure \ref{fig:cubical}. This labeling shows that in fact $\S_{2,4}$ is the CAT(0) cube complex in that figure.
Ardila, Bastidas, Ceballos, and Guo \cite{Ard17} proved that this is an instance of a general phenomenon:

\begin{theorem} \label{thm:ABCG} \cite{Ard17}
For any positive integers $m$ and $n$, the cube complex $\S_{m,n}$ of monotone paths of length $n$ in a strip of height $m$ is  CAT(0).  
\end{theorem}

%\subsection{\textsf{The cubical complex $S_{m,n}$.}}

%The graph $S_{m,n}$ of arm configurations connected by the local moves above was shown to be a CAT(0) cubical complex in \cite{Ard17}, which means that it satisfies a certain discrete curvature condition. In that paper, the authors also show the cubical complex can be described by a poset together with a set of  inconsistent pairs (PIP) (see Figure \ref{fig:cubical}, where red dotted lines indicate the inconsistent pairs). A downset in a poset is a set of elements closed downward under the relation of the order. The downsets of the PIP are not allowed to contain any inconsistent pairs. It has been shown that the downsets of a certain family of PIPs are in bijective correspondence with the arm configurations $\Omega_{m,n}$. The arm configurations corresponding to the downsets of the PIP in Figure \ref{fig:cubical} are shown in  Figure \ref{fig:configs}.
%In the cubical complex, moving from one configuration to an adjacent one by moving one joint or the endpoint corresponds to adding or deleting one element of a downset in the PIP.
%

\begin{figure}[h]
\begin{center}
\begin{tabular}{rcrcrc}
$\emptyset$ & 
\begin{tikzpicture}
\draw[step=.5cm,gray] (-2,-1) grid (0,0);
\draw[ultra thick, -, blue] (-2,-1) -- (-1.5,-1) -- (-1,-1) -- (-.5,-1) -- (0,-1);
\end{tikzpicture} &
1&
\begin{tikzpicture}
\draw[step=.5cm,gray] (-2,-1) grid (0,0);
\draw[ultra thick, -, blue] (-2,-1) -- (-1.5,-1) -- (-1,-1) -- (-.5,-1) -- (-.5,-.5);
\end{tikzpicture} &
12 & 
\begin{tikzpicture}
\draw[step=.5cm,gray] (-2,-1) grid (0,0);
\draw[ultra thick, -, blue] (-2,-1) -- (-1.5,-1) -- (-1,-1) -- (-1,-.5) -- (-.5,-.5);
\end{tikzpicture} \\
123&
\begin{tikzpicture}
\draw[step=.5cm,gray] (-2,-1) grid (0,0);
\draw[ultra thick, -, blue] (-2,-1) -- (-1.5,-1) -- (-1,-1) -- (-1,-.5) -- (-1,-0);
\end{tikzpicture} &
124&
\begin{tikzpicture}
\draw[step=.5cm,gray] (-2,-1) grid (0,0);
\draw[ultra thick, -, blue] (-2,-1) -- (-1.5,-1) -- (-1.5,-.5) -- (-1,-.5) -- (-.5,-.5);
\end{tikzpicture} &
1234 & 
\begin{tikzpicture}
\draw[step=.5cm,gray] (-2,-1) grid (0,0);
\draw[ultra thick, -, blue] (-2,-1) -- (-1.5,-1) -- (-1.5,-.5) -- (-1,-.5) -- (-1,0);
\end{tikzpicture} \\
1245 & 
\begin{tikzpicture}
\draw[step=.5cm,gray] (-2,-1) grid (0,0);
\draw[ultra thick, -, blue] (-2,-1) -- (-1.5,-1) -- (-1.5,-.5) -- (-1,-.5) -- (-1,-1);
\end{tikzpicture} &
1247&
\begin{tikzpicture}
\draw[step=.5cm,gray] (-2,-1) grid (0,0);
\draw[ultra thick, -, blue] (-2,-1) -- (-2,-.5) -- (-1.5,-.5) -- (-1,-.5) -- (-.5,-.5);
\end{tikzpicture} &
12346&
\begin{tikzpicture}
\draw[step=.5cm,gray] (-2,-1) grid (0,0);
\draw[ultra thick, -, blue] (-2,-1) -- (-1.5,-1) -- (-1.5,-.5) -- (-1.5,0) -- (-1,0);
\end{tikzpicture} \\
12347 & 
\begin{tikzpicture}
\draw[step=.5cm,gray] (-2,-1) grid (0,0);
\draw[ultra thick, -, blue] (-2,-1) -- (-2,-.5) -- (-1.5,-.5) -- (-1,-.5) -- (-1,0);
\end{tikzpicture} &
12457& 
\begin{tikzpicture}
\draw[step=.5cm,gray] (-2,-1) grid (0,0);
\draw[ultra thick, -, blue] (-2,-1) -- (-2,-.5) -- (-1.5,-.5) -- (-1,-.5) -- (-1,-1);
\end{tikzpicture} &
123467&
\begin{tikzpicture}
\draw[step=.5cm,gray] (-2,-1) grid (0,0);
\draw[ultra thick, -, blue] (-2,-1) -- (-2,-.5) -- (-1.5,-.5) -- (-1.5,0) -- (-1,0);
\end{tikzpicture} \\
124578 & 
\begin{tikzpicture}
\draw[step=.5cm,gray] (-2,-1) grid (0,0);
\draw[ultra thick, -, blue] (-2,-1) -- (-2,-.5) -- (-1.5,-.5) -- (-1.5,-1) -- (-1,-1);
\end{tikzpicture} &
1234679&
\begin{tikzpicture}
\draw[step=.5cm,gray] (-2,-1) grid (0,0);
\draw[ultra thick, -, blue] (-2,-1) -- (-2,-.5) -- (-2,0) -- (-1.5,0) -- (-1,0);
\end{tikzpicture} &
123467910 & 
\begin{tikzpicture}
\draw[step=.5cm,gray] (-2,-1) grid (0,0);
\draw[ultra thick, -, blue] (-2,-1) -- (-2,-.5) -- (-2,0) -- (-1.5,0) -- (-1.5,-.5);
\end{tikzpicture} \\
 
\end{tabular}\\
\end{center}
\caption{The monotone paths of length $4$ in a tunnel of height $2$.}
\label{fig:configs}
\end{figure}
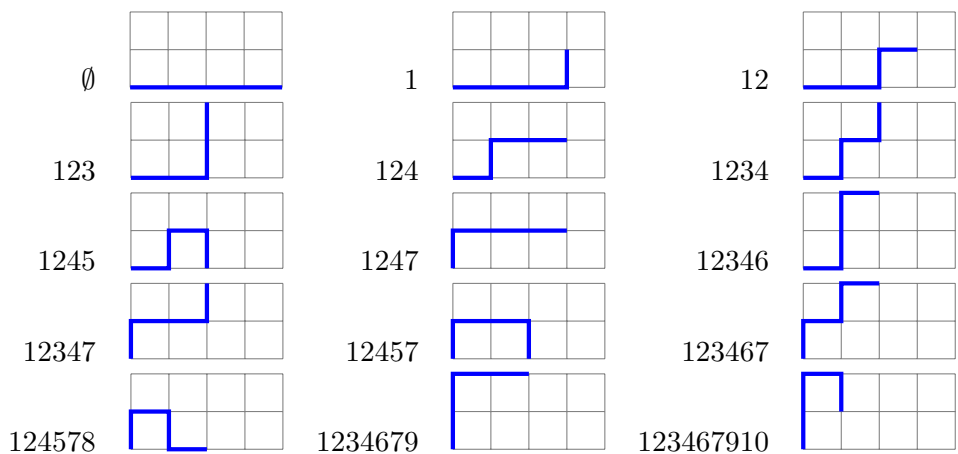

\subsubsection{\textsf{The coral PIP}}

To prove that $\S_{m,n}$ is  CAT(0),  \cite{Ard17} introduced and used the technique shown in Section \ref{sec:CAT(0)-PIP}: the authors introduced the \emph{coral PIP} $C_{m,n}$, and showed that its corresponding CAT(0) cube complex is $\S(C_{m,n}) \cong \S_{m,n}$. To describe the coral PIP, we need to introduce the notion of a \emph{coral snake}.

\begin{definition} \label{def:coralsnake}
A \emph{coral snake} $\lambda$ of height at most $m$ is an oriented path of unit squares, colored alternatingly black and red (starting with black), inside the strip of height~$m$ such that:
\begin{enumerate}[(i)]
\item The snake $\lambda$ starts with the bottom left square of the strip, and takes unit steps up, down, and right.
\item Suppose $\lambda$ turns from a vertical segment $V_1$ to a horizontal segment $H$ to a vertical segment $V_2$ at corners $C_1$ and $C_2$. Then $V_1$ and $V_2$ face the same direction if and only if $C_1$ and $C_2$ have the same color. \label{defcondition:corner}
(Note: we consider the first column of the snake a vertical segment going up, even if it consists of a single cell.) 
\end{enumerate}

\begin{figure}[htbp]
	\centering
		\includegraphics[height = 1.5cm]{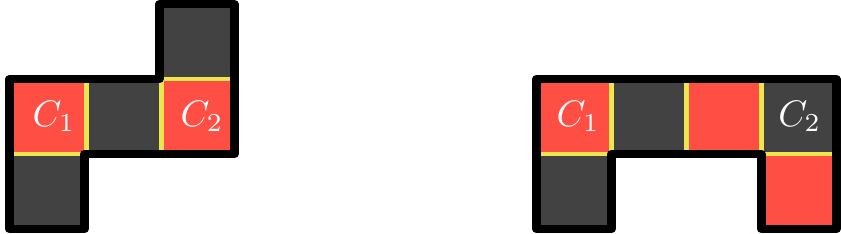}
	\caption{Illustration of condition~\eqref{defcondition:corner} in Definition~\ref{def:coralsnake}.}
	\label{fig:conditioncoralsnakeii}
\end{figure}

The \emph{length} $l(\lambda)$ is the number of unit squares of $\lambda$, the \emph{height} $h(\lambda)$ is the number of  rows it touches, and the \emph{width} $w(\lambda)$ is the number of columns it touches. We say that~$\mu$ \emph{contains}~$\lambda$, in which case we write $\lambda \preceq \mu$, if $\lambda$ is an initial sub-snake of $\mu$ obtained by restricting to the first $k$ cells of $\mu$ for some $k$. We write $\lambda \prec \mu$ if $\lambda \preceq \mu$ and $\lambda \neq \mu$.
For technical reasons, sometimes we will also consider the empty snake $\emptyset$.
\end{definition}

\begin{figure}[htbp]
	\centering
		\includegraphics[height = 1.5cm]{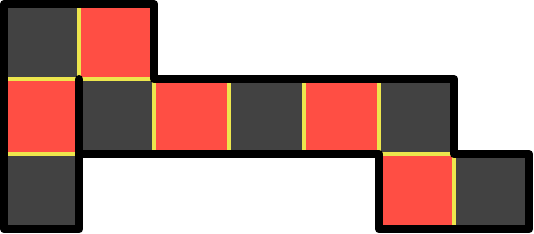} 
		\caption{A coral snake $\lambda$ of length $l(\lambda) = 11$, height $h(\lambda) = 3$, and width $w(\lambda) = 7$. We encourage the reader to check condition $(ii)$. 
		}
	\label{fig:snake}
\end{figure}

Although the colors of a coral snake are a useful visual aid, we sometimes omit them since they are uniquely determined by the shape of the snake.

\begin{definition}\label{def:snakePIP}
Define the \emph{coral PIP} $C_{m,n}$ as follows:

\noindent
$\bullet$ Elements: pairs $(\lambda,s)$ consisting of a coral snake $\lambda \neq \emptyset$ with $h(\lambda) \leq m$ and an integer $0 \leq s \leq n-l(\lambda)-w(\lambda)+1$.

% non-negative integer $s$ with $s \leq n-l(\lambda)-w(\lambda)+1$.

\noindent
$\bullet$ Order: 
$(\lambda,s) \leq (\mu,t)$ if $\lambda \preceq  \mu$ and $s \geq t$.

\noindent
$\bullet$ Inconsistency: $(\lambda,s)  \nleftrightarrow (\mu,t)$ if neither $\lambda$ nor $\mu$ contains the other.

For simplicity, we call the elements $(\lambda,s)$ of the coral PIP \emph{numbered snakes}. We write them by placing the number $s$ in the first cell of $\lambda$.
\end{definition}

The left panel of Figure \ref{fig:PIP9_withpaths} shows the coral PIP $C_{2,9}$. The right panel shows how  $C_{2,8}, C_{2,7}, \ldots$ are subPIPs of $C_{2,9}$; they are obtained from it by removing one colored layer at a time.

\begin{theorem}\cite{Ard17}\label{thm:ABCG2}
The monotone paths of length $n$ in a strip of height $m$ are in bijection with the order ideals of the coral PIP $C_{m,n}$.
\end{theorem} 

\begin{figure}[tbh]
	\centering
		\includegraphics[height = 8cm]{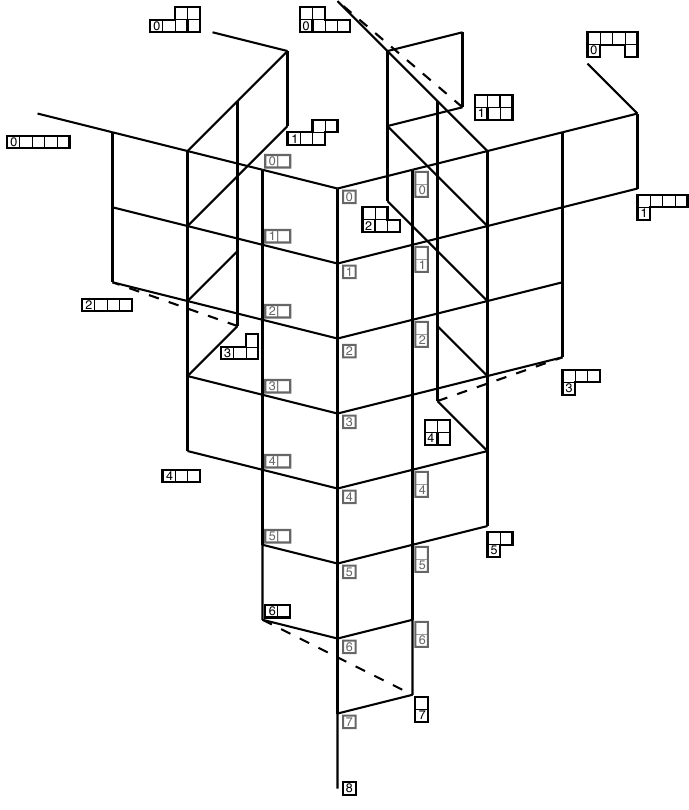}   \qquad 
		\includegraphics[height = 8cm]{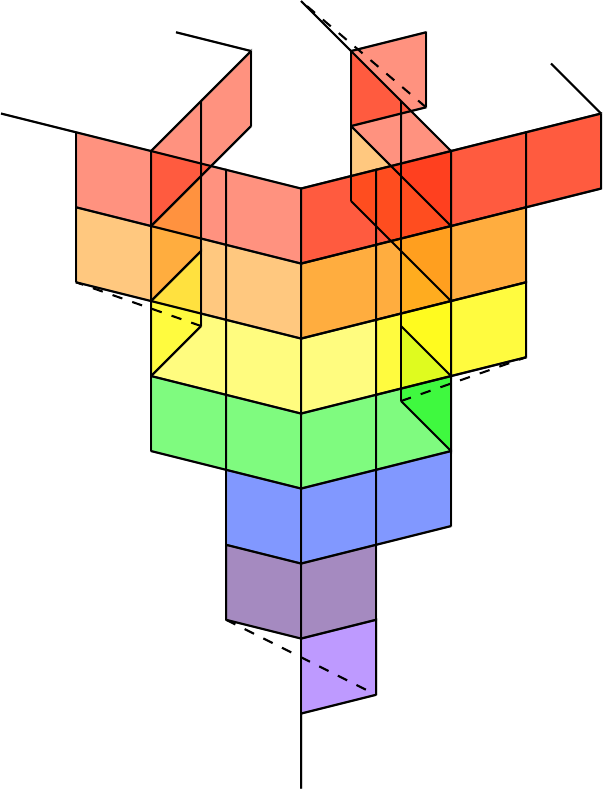}  
		\caption{(a) The coral PIP $C_{2,9}$ with its numbered snakes. (b) The coral PIPs $C_{2,8}, C_{2,7}, \ldots$ are obtained from $C_{2,9}$ by removing one colored layer at a time.}
	\label{fig:PIP9_withpaths}
\end{figure}

\subsubsection{\textsf{Using the coral PIP to find a small bottleneck in $\S_{m,n}$}}

We now use Bottleneck Lemma \ref{lem:bottleneck} to find a small bottleneck in $\S_{m,n}$. Figure \ref{fig:PIP9_withpaths} shows us the way: there is a natural choice of an inconsistent pair in the coral PIP $C_{m,n}$.

\begin{definition}
Let the \emph{low inconsistent pair} of the coral PIP $C_{m,n}$ consist of:
\[
a=
\begin{ytableau}
 \\
\scriptstyle n-2 
\end{ytableau} \,\,\,  ,
\qquad
b=
\begin{ytableau}
\none & \none \\
\scriptstyle n-3 &  
\end{ytableau} \,\,\, .
\]
Let the \emph{backbone} of $C_{m,n}$ be the consistent order ideal
\[
B = 
\left\{ \, 
\begin{ytableau}
\scriptstyle \tiny{n-1}
\end{ytableau}
<
\begin{ytableau}
\scriptstyle n-2
\end{ytableau}
<
\cdots
< 
\begin{ytableau}
\scriptstyle 0
\end{ytableau}
\,
\right\} \subset C_{m,n}.
\]
\end{definition}

Notice that $a$ and $b$ have the largest possible indices among the tableaux of their respective shapes. This inconsistent pair decomposes the set of consistent order ideals into three subsets $\S(C_{m,n}) = \S(C_{m,n})_{a \nleftrightarrow b} \sqcup \S(C_{m,n})_{a} \sqcup \S(C_{m,n})_{b}$, depending on whether an order ideal contains neither $a$ nor $b$, only $a$, or only $b$. 

As an example, in Figure \ref{fig:cubical}, $a=3$, $b=5$, and 
\begin{eqnarray*}
\S(C_{2,4})_{3 \nleftrightarrow 5} &=& \{\emptyset, 1, 12, 124, 1247\}, \\ 
\S(C_{2,4})_{3} &=& \{123, 1234, 12347, 12346, 123467, 1234679, 12346710\}, \\
\S(C_{2,4})_{5} &=& \{1245, 12457, 124578\}.
\end{eqnarray*}
Each one of these order ideals corresponds to a path shown in Figure \ref{fig:configs}. The reader may wish to compare the lists above with those paths, to motivate the following lemmas.

%
%
%$\{\emptyset, 1, 12, 124, 1247\}$ which are precisely the order ideals that don't contain either of the elements of the lowest inconsistent pair ($3$ and $5$) of $C_{2,4}$. The set $T=\{123, 1234, 12347, 12346, 123467, \\1234679, 12346710\}$ consists of the set of ideals containing the element 3. The set $U=\{1245, 12457, \\ 124578\}$ consists of the set of ideals containing the element 5.

\bigskip

The statements of the following two lemmas can be understood independently, but the proofs assume familiarity with the details of the bijection $\S(C_{m,n}) \rightarrow \S_{m,n}$ of Theorem \ref{thm:ABCG2}, as explained in Lemma 5.9 and Theorem 5.5 of \cite{ABCG}.

\begin{lemma}\label{lemma:S}
The bijection $\S(C_{m,n}) \rightarrow \S_{m,n}$ of Theorem \ref{thm:ABCG2} restricts to a bijection
\[
\S(C_{m,n})_{a \nleftrightarrow b} \rightarrow \text{ paths containing at most one vertical step}.
\]
There are $n+1$ such paths.
\end{lemma}

\begin{proof}
We claim that the order ideals $I$ in $\S(C_{m,n})_{a \nleftrightarrow b}$ are precisely those that only contain numbered snakes of length $1$. 
Indeed, assume $I \in \S(C_{m,n})_{a \nleftrightarrow b}$. We claim that $I$ only contains numbered snakes of length $1$. 
Indeed, if $I$ contained a numbered snake $(\lambda,s)$ of length at least $2$, then it would also contain the numbered snake $(\lambda_2,s) \leq (\lambda, s)$, where $\lambda_2$ consists of the first two squares of $\lambda$. If $\lambda_2$ consists of two vertically arranged squares, then $s \leq n-2-1+1 = n-2$ by the definition of $C_{m,n}$, so $I$ also contains $a = (\lambda_2, n-2) \leq (\lambda_2,s)$, a contradiction. Similarly, if $\lambda_2$ consists of two horizontally arranged squares, then $s \leq n-2-2+1 = n-3$, so $I$ also contains $b = (\lambda_2, n-3) \leq (\lambda_2,s)$, a contradiction.
The converse is clear.

It follows that $\S(C_{m,n})_{a \nleftrightarrow b}$ consists of the empty ideal -- which corresponds to the horizontal path -- and for
 $0 \leq k \leq n-1$, the order ideal generated by $(\square, k)$, which corresponds to the path that takes $k$ horizontal steps, then one vertical step, and then $n-1-k$ horizontal steps, as desired.
\end{proof}

\begin{lemma} \label{lemma:U}
The bijection $\S(C_{m,n}) \rightarrow \S_{m,n}$ of Theorem \ref{thm:ABCG2} restricts to a bijection
\[
\S(C_{m,n})_{b} \rightarrow \text{ paths that contain a second vertical step, which points down}.
\]
The number of such paths is
\[
u_m(n) := c_m(n-4) + 2c_m(n-5) + 3c_m(n-6) + \cdots + (n-3)c_m(0).
\]
\end{lemma}

\begin{proof}
First consider an order ideal $I$ in $\S(C_{m,n})_{b}$. Then $b=(\square\hspace{-.05cm}\square, n-2) \in I$, so
%and we let $k$ be the smallest number such that $(\square\hspace{-.05cm}\square, k) \in I$. Then 
the first two steps of the \emph{coral snake tableaux} $T$ corresponding to $I$ in \cite[Lemma 5.9]{ABCG} are horizontally arranged.
%, say:
%\[
%I 
%\longleftrightarrow 
%\begin{ytableau}
% j & k  
%\end{ytableau}  \,\,\cdots
%%\leftrightarrow
%%\text{draw the path}
%\]
%for some $0 \leq j<k \leq n-1$. 
In the corresponding monotone path, the first and second vertical steps %are in positions $j+1$ and $k+1$, and they 
point up and down, respectively. 

Conversely, consider a monotone path $P$ whose first two vertical steps are the $j$th and $k$th, which  point up and down, respectively for some $1 \leq j < k \leq n$. The corresponding coral snake tableaux $T$ and its decomposition into join-irreducibles -- as described in the proof of \cite[Theorem 5.5]{ABCG}, 
starts:
\[
T = 
\begin{ytableau}
\scriptstyle j-1 & \scriptstyle k-2  
\end{ytableau} \cdots 
= 
\begin{ytableau}
\scriptstyle j-1 
\end{ytableau} 
\vee
\begin{ytableau}
\scriptstyle k-3 & \scriptstyle k-2  
\end{ytableau} 
\vee \cdots
\]

%
%,
%so in the decomposition $T=T_1 \vee T_2 \vee \cdots$ we have that 
%\[
%T_2 = 
%\begin{ytableau}
%\scriptstyle k-3 & \scriptstyle k-2  
%\end{ytableau},
%\]
The second join-irreducible above corresponds to element $(\square\hspace{-.05cm}\square, k-3)$ of $C_{m,n}$, which must be in the ideal $I$ corresponding to $P$. But then
$(\square\hspace{-.05cm}\square, k-3) \geq (\square\hspace{-.05cm}\square, n-3) = b$ implies that the  ideal $I$ contains $b$ as well, as desired.

Now let us count the number of arm configurations for a given choice of $k \geq 3$. There are $k-2$ choices for the value of $j$, and this choice entirely determines the first $k$ steps of the path. Step $k+1$ must be horizontal, and there are $c_m(n-k-1)$ choices for the rest of the path. It follows that $\displaystyle u_m(n) = \sum_{k \geq 3} (k-2)c_m(n-k-1)$, as desired.
\end{proof}

%
%%
%%\subsection{\textsf{Using the coral PIP to bound the conductance of the Markov chain}}%
%These combinatorial considerations allows us to bound the %conductance of the Markov chain by bounding the 
%bottleneck ratio for the set $\S(C_{m,n})_{b}$. 
%%

These combinatorial considerations allows us to show that the small bottleneck of $n+1$ vertices in $\S(C_{m,n})_{a \nleftrightarrow b}$ separates the cube complex $\S(C_{m,n})$ into two parts $\S(C_{m,n})_{a}$ and $\S(C_{m,n})_b$ that each contain roughly a constant fraction of the (exponentially many) vertices.

\begin{lemma}\label{lemma:bottleneck}
For any fixed height $m \geq 2$ of the strip, there is a constant $0<C_m < 1$ such that
\[
\lim_{n \rightarrow \infty} \frac{\S(C_{m,n})_{b}}{\S(C_{m,n})} = 
\lim_{n \rightarrow \infty} \frac{u_m(n)}{c_m(n)} = C_m.
\]
The constant $C_m$ decreases with $m$, and approaches $1.5 - \sqrt2 \approx 0.0858\ldots$ as $m$ goes to infinity.
\end{lemma}

\begin{proof} 
We saw that the generating function for $c_m(n)$ is rational
\[
\sum_{n \geq 0} c_m(n) x^n = \frac{p(x)}{q(x)}
\]
and Lemma \ref{lemma:U} shows that
\[
\sum_{n \geq 0} u_m(n) x^n = 
(\sum_{n \geq 0} u_m(n) x^n) (x^4 + 2x^5+3x^6+ \cdots) 
= 
\frac{p(x)}{q(x)} \frac{x^4}{(1-x)^2}.
\]
Furthermore, as discussed in Corollary \ref{cor:c_mPerron}, the denominators $q(x)$ and $q(x)(1-x)^2$ have a simple factor $1-r_mx$ that dominates the others, where $r_m > 1$ is the Perron eigenvalue of $A_m$, so that 
Corollary \ref{cor:c_mPerron} then tells us that 
\[
c_m(n) \sim q_m \cdot r_m^n \qquad \text{and} \qquad  u_m(n) \sim v_m \cdot r_m^n, \qquad 
\]
where 
\[
q_m=\frac{-{r_m} p(1/r_m)}{q'(1/r_m)}
\]
and
\begin{eqnarray*}
v_m & = & \frac{-r_m [p(x)x^4]_{x=1/r_m}}{[q(x)(1-x)^2]'|_{x=1/r_m}} \\
&=& \frac{-r_m p(1/r_m)}{q'(1/r_m)(1-\frac{1}{r_m})^2 \, r_m^4.}
\end{eqnarray*}
taking into account that $q(1/r_m)=0$.
It follows that
\[
\lim_{n \rightarrow \infty} \frac{u_m(n)}{c_m(n)} = \frac{v_m}{q_m} = \frac{1}{r_m^2(r_m-1)^2}.
\]
Since the Perron eigenvalues $r_m$ increase starting at $1.618\ldots$ and converging to $1+\sqrt{2}$ by Proposition \ref{prop:Perron}, the constants $C_m = \frac{1}{r_m^2(r_m-1)^2}$ decrease starting at $0.345\ldots$ and converge to $1.5 - \sqrt2$ as desired.
\end{proof}

\section{\textsf{The Markov chains of monotone paths in a strip mix slowly.}}  \label{sec:Markov}

We are finally ready to turn to the main goal of this paper: to investigate the mixing times of two natural Markov chains on the set of monotone paths of length $n$ in a strip of height $m$, based on the local moves introduced in Section \ref{sec:intro}. We can think of these as random walks on the graph $S_{m,n}$.

\subsection{\textsf{Preliminaries on Markov chains}\label{sec:MC}}

Let us review some basic facts about the mixing of Markov chains; for a more detailed account, see for example \cite{Levin09}.

\begin{definition} \label{Markov-Chain}
A sequence of random variables $(X_0,X_1, ...)$ is a \emph{finite Markov chain} with finite state space $\Omega$ and \emph{transition matrix} $P$ if for all $x,y \in \Omega$, all $t \geq 1$, and all events $H_{t-1} = \bigcap_{0 \leq s \leq t-1} (X_s = x_s)$ satisfying \textbf{P}($H_{t-1} \cap (X_t = x)) > 0$, we have 
\[
\mathbf{P}(X_{t+1} = y  \, | \, H_{t-1} \cap (X_t = x)) = \mathbf{P}(X_{t+1} = y \, | \, X_t = x) = P(x,y).
\]
\end{definition}

In words, the Markov property requires that the probability $P(x,y)$ of moving from the current state $X_t=x$  to the next state $X_{t+1}=y$ does not depend on any earlier states. %We call $(\mathbf{P}(x,y))_{x,y \in \Omega}$ the \emph{transition matrix} of the Markov chain.

A Markov chain is \emph{irreducible} if for any $x, y \in \Omega$ there exists an integer $r \ge 0$ such that $P^r(x,y) > 0$. The \emph{period} of a state $x$ is the greatest common divisor of the \emph{return times} $t$ such that $P^t(x,x)>0$. A Markov chain is \emph{aperiodic} if every state $x$ has period $1$. 
%It is \emph{symmetric} if $P(x,y) = P(y,x)$ for all $x,y \in \Omega$.

\begin{definition}\label{def:stationary}
A probability distribution $\pi$ over $\Omega$ is a \emph{stationary distribution} for a Markov chain on $\Omega$ with transition matrix $P$  if $\pi=\pi P$. 
\end{definition}

\begin{theorem} \label{thm:stationary}
If a Markov chain with transition matrix $P$ is irreducible and aperiodic then it has a unique stationary distribution $\pi$. We have, for all $x,y \in \Omega$,
\[
\lim_{t \to \infty} P^t(x,y) = \pi(y).
\]
%, as $t \to \infty$, $P^t(x,y) \to \pi(y)$. 
%In addition, if the chain is symmetric, the stationary distribution is uniform over $\Omega$; that is, $\pi(x) = 1/|\Omega|$ for all $x \in \Omega$.
\end{theorem}

\begin{propdef}\cite[Proposition 1.19]{Levin09}
Suppose a Markov chain on $\Omega$ has transition  matrix $P$. If a probability distribution $\pi$ on $\Omega$ satisfies
\[
\pi(x)P(x,y) = \pi(y) P(y,x) \ \ \ \ \ \forall x,y \in \Omega
\]
the chain is said to be \textit{reversible} with respect to $\pi$, and moreover, $\pi$ is a stationary distribution for the Markov chain.
\end{propdef}

Consequently, if a Markov chain is irreducible, aperiodic, and reversible with respect to a distribution $\pi$, then $\pi$ is the unique stationary distribution for the chain. 

In order to quantify how quickly the chain converges to stationarity, we introduce the notions of total variation distance and mixing time.

%\begin{enumerate}
%\item
%If the Markov chain is irreducible, the stationary distrubition is unique.
%\item
%If in addition the Markov chain is aperiodic, then the $t$-fold distribution $P^t(x,\cdot)$ converges to $\pi(\cdot)$ for any $x \in \Omega$. 
%\item
%If in addition the Markov chain is symmetric, then the stationary distribution is the uniform distribution, where every state is equally likely.
%\end{enumerate}

\begin{propdef} \cite[Proposition 4.2]{Levin09} \label{total-variation}
The \emph{total variation distance} between two probability distributions $\mu$ and $\nu$ on a state space $\Omega$ is:
\begin{eqnarray*}
||\mu - \nu||_{TV} &:=& \max_{A \subseteq \Omega} |\mu(A) - \nu(A)| \\
&=& \frac12 \sum_{x \in \Omega} | \mu(x) - \nu(x)|.
\end{eqnarray*}
\end{propdef}

%The first definition shows that $||\mu - \nu||_{TV}$ is the largest possible difference in probabilities for an event $A \subseteq \Omega$, while the second definition shows it is simply the half sum of the probability differences of the individual states $x \in \Omega$.

\begin{definition} \label{Markov-total-variation-distance}
For a Markov chain $(X_0,X_1, ...)$ we define
\[
d(t) \coloneqq \max\limits_{x \in \Omega} || P^t(x, \cdot) - \pi(\cdot) ||_{TV}
\]
%to be the largest distance between the $t$-step distribution $P^t(x, \cdot)$ starting at a state $x$ and the uniform distribution $\pi(\cdot)$. 
The \emph{$\varepsilon$-mixing time} of the Markov chain is defined to be 
\[
\tau(\varepsilon) \coloneqq \min \{t : d(t) \leq \varepsilon \}.
\]
\end{definition}

The distance $d(t)$ measures how far the farthest $t$-step distribution $P^t(x, \cdot)$ starting at an $x \in \Omega$ is from the stationary distribution $\pi$. The mixing time tells us how many steps we need to take until the $t$-step distribution starting at any state is within $\varepsilon$ of the stationary distribution, measured in total variation distance.

%It can be technically convenient to analyze the \textit{lazy} version of a Markov chain. The lazy version of a chain $M$ is given by $\frac 12 (I+M)$. That is, there is a self loop of probability $1/2$ at every state and the original transition kernel is followed with probability $1/2$. This has the effect of slowing the convergence of the chain by at most a factor of 2, and it ensures that all the eigenvalues of the Markov chain are non-negative.

%
%We say the Markov chain \emph{mixes polynomially} if the mixing time can be upper bounded by a polynomial in $n$ and  \emph{mixes slowly} if the mixing time is bounded below by an exponential function in $n$. \red{What is $n$ here?}

The bottleneck ratio of a Markov chain is a geometric quantity that can provide upper as well as lower bounds on the mixing time. In our case, it is the lower bound which is relevant.

%\begin{definition}\label{conductance}
%The \emph{bottleneck ratio} or \emph{conductance} of an irreducible, aperiodic Markov chain on $\Omega$ with transition matrix $P$ is given by
%\begin{center}
% $\Phi \defeq \displaystyle \min\limits_{S\subset\Omega: 0<\pi(S)\leq\frac{1}{2}} \displaystyle\frac{\sum\limits_{x\in S,y\notin S} \pi(x)P(x,y)}{\pi(S)}.$
% \end{center}
%\end{definition}
\begin{definition}\label{conductance}
The \emph{bottleneck ratio} or \emph{conductance} of an irreducible and aperiodic Markov chain on $\Omega$ with transition matrix $P$ is given by
\begin{center}
 $\Phi \defeq \displaystyle \min\limits_{S\subset\Omega: 0<\pi(S)\leq\frac{1}{2}} \displaystyle\frac{\sum\limits_{x\in S,y\notin S} \pi(x)P(x,y)}{\pi(S)}.$
 \end{center}
\end{definition}

In what follows in the rest of this section, we will always assume that the Markov chain in question is irreducible and aperiodic. Let the eigenvalues of its transition matrix be given by $1 = \lambda_0 > \lambda_1 \ge \cdots \ge \lambda_{|\Omega| - 1} > -1$. Let $\lambda_{\max} = \max\{\lambda_1,|\lambda_{|\Omega|-1}|\}$. A lower bound in the mixing time can be established in two steps.  The total variation distance to stationarity can be shown to be lower bounded below by the second largest eigenvalue in magnitude. 

\begin{theorem}\cite[Theorem 4.9]{MonTet06}  \label{thm:mixtime-lb-lambda} The mixing time of an irreducible Markov chain can be bounded as 
\[
\tau(\varepsilon) \ge \frac{|\lambda_{max}|}{1-|\lambda_{max}|} \ln((2\varepsilon)^{-1}).
\]

\end{theorem}

Secondly, the second largest eigenvalue in magnitude can be related to the conductance of the chain.

\begin{theorem}\cite[Lemma 2.6]{Sin93} For an irreducible and reversible Markov chain with conductance $\Phi$,
\[
\lambda_{1} \ge 1-2\Phi.
\]

\end{theorem}

Since $\lambda_{\max} \ge \lambda_1$, the following is an immediate consequence.

\begin{corollary}\label{cor:lambda-lb-conductance}
For an irreducible and reversible Markov chain with conductance $\Phi$,
\[
\lambda_{max} \ge 1-2\Phi.
\]

\end{corollary}

Combining Theorem \ref{thm:mixtime-lb-lambda} and Corollary \ref{cor:lambda-lb-conductance} results in the following lower bound on the mixing time.
\begin{theorem}\label{thm:slow-mixing}
For an irreducible and reversible Markov chain with conductance $\Phi$,
\[
\tau(\varepsilon) \geq \frac{1-2\Phi}{2\Phi}\ln((2\varepsilon)^{-1}).
\]
\end{theorem}

Therefore, to show that a Markov chain mixes slowly, it suffices to find a set $S$ whose bottleneck ratio is very small. 
%It follows that if one can find a set $S$ whose bottleneck ratio is exponentially small, one obtains an exponentially small upper bound on the conductance.

%
%Clearly, the kernel of this Markov chain is the graph $S_{m,n}$. By standard results in Markov chain theory (see Section \ref{sec:MC}), $M_{m,n}$ converges to the uniform distribution over the configurations in $\Omega_{m,n}$. 
%
%
%Our paper has three main contributions.
%
%1. ENUMERATION
%
%2. BOTTLENECK
%
%3. MARKOV CHAIN
%Our main probabilistic result is that the Markov chain $M_{m,n}$ mixes slowly, taking exponential time in $n$ to come close in total variation distance to the uniform distribution. A precise statement is given in Theorem \ref{thm:main}. 
%%We believe the result to hold in general when the height of the strip $m$ is any constant.
%

\subsection{\textsf{The symmetric Markov chain of monotone paths in a strip .}}

\begin{definition} \label{robotic-arm-markov-chain}
\textbf{(The symmetric Markov chain $M_{m,n}$ on monotone paths)} 
Let $\Omega_{m,n}$ denote the set of monotone paths of length $n$ in a strip of height $m$. 
For $t \geq 0$, if $X_t \in \Omega_{m,n}$ is the path at time $t$, we obtain the next path $X_{t+1}$ as follows. 
 \begin{enumerate}
%\item With probability $\frac 12$, do nothing and let $X_{t+1} = X_t$. 
\item
Choose any of the $n+1$ vertices of the path uniformly at random, except the first one, so that each possibility has probability $\frac1n$. 
%\item
%Uniformly at random, either choose the $j$th interior joint between the $j$th and $j+1$st link for $1 \le j \le n-1$ or the endpoint of the arm, so that each possibility has probability $\frac1n$. 
\item 
\begin{enumerate}
\item Suppose an interior vertex $j$ is chosen. If $X_t$ has a corner at vertex $k$, and the corner can be switched, switch that corner to get $X_{t+1}$. If it does not, let $X_{t+1} = X_t$.
\item Suppose the last vertex is chosen. If the last step is N or S, flip it to E with probability $\frac12$ and do nothing with probability $\frac12$. If the last step is E, flip it to S with probability $\frac12$ and to N with probability $\frac12$.
\end{enumerate}
%\begin{enumerate}
%\item If an interior joint $j$ is chosen, and there is a valid corner switch move, then carry it out to get $X_{t+1}$ and if the move is invalid, $X_{t+1} = X_t$.
%\item If the endpoint is chosen, then flip it $90$ degrees clockwise with probability $\frac 12$ and counterclockwise with probability $\frac 12$, as long as the resulting move is valid. If the move is not valid, $X_{t+1} = X_t$.
%\end{enumerate}
\end{enumerate}
\end{definition}

\begin{theorem}\label{slow-mixing-theorem}
Let $m \geq 2$ be a fixed integer. %Let $M_{m,n}$ be the lazy Markov chain on monotone paths of length $n$ in a strip of height $m$. 
The stationary distribution of the symmetric Markov chain $M_{m,n}$ is uniform.
The mixing time of the chain $M_{m,n}$ grows exponentially with $n$; explicitly, there exists a constant $C>0$ such that for $n$ sufficiently large, the mixing time $\tau(\varepsilon)$ satisfies
\[
\tau(\varepsilon)  \ge C \cdot r_m^n \ln(\varepsilon^{-1})
\]
where $r_m > 1$ is the Perron eigenvalue of $A_m$ in Lemma \ref{lem:Perron}.
\end{theorem}

\begin{proof}
To prove the first statement, we verify that our Markov chain $M_{m,n}$ on monotone paths has all the properties of Theorems \ref{thm:stationary} and \ref{thm:slow-mixing}. 
It is irreducible since there is a set of moves that allows us to go between any two arm configurations: one can start with the downset corresponding to the starting configuration, remove elements one at a time until we are at the empty downset, and then add elements back in to get to the final configuration.
The chain is aperiodic since it contains states that are connected by one move to themselves. 
Finally, the chain is reversible with respect to the uniform distribution over monotone paths: for every pair of states $x,y$ connected by a corner flip, $P(x,y) = P(y,x) = \frac{1}{n}$ while for a pair of states connected by flipping the end, $P(x,y) = P(y,x) = \frac{1}{2n}$. 
It follows that the stationary distribution is uniform.

To bound the mixing time, recall that in Lemma \ref{lemma:bottleneck} we identified a small bottleneck of $n+1$ vertices in $\S_{a \nleftrightarrow b} := \S(C_{m,n})_{a \nleftrightarrow b}$ that separates the transition kernel $\S=\S(C_{m,n})$ of our Markov chain into two parts $\S_{a}$ and $\S_b$ that each contain roughly a constant fraction of the (exponentially many) vertices. This gives us that
\begin{eqnarray*}
\Phi & \le & \displaystyle \frac{\sum\limits_{x \in \S_b, y \notin \S_b} \pi(x)P(x,y)}{\pi(\S_b)} \\
& = & \displaystyle \frac{\sum\limits_{x \in \S_b, y \in \S_{a \nleftrightarrow b}}  P(x,y)}{|\S_b|} 
\end{eqnarray*}
since an edge $xy$ that leaves $\S_b$ can only arrive in $\S_{a \nleftrightarrow b}$, and $\pi(x) = 1/{c_m(n)}$ for all $x$. Now, each $y \in  \S_{a \nleftrightarrow b}$ is connected to at most one $x \in  \S_b$: if $y$ corresponds to an order ideal $I$ of the coral PIP $P_{m,n}$ that does not contain $a$ or $b$, then $x$ can only correspond to the order ideal $I \cup b$ that contains $b$ -- if $I \cup b$ is indeed an order ideal. It follows that
\begin{eqnarray*}
\Phi & \leq & \displaystyle \frac{|\S_{a \nleftrightarrow b}| \cdot \frac{1}{n}}{|\S_b|}  \\ 
& \sim & \frac 1{C_m q_m \cdot r_m^n} 
%& \sim & \frac 1n \cdot \frac{(n+1)} {C_ma_mr_m^n} 
\end{eqnarray*}
by Lemmas \ref{lemma:S} and \ref{lemma:bottleneck} and Corollary \ref{cor:c_mPerron}. The desired bound on the mixing time then follows from Theorem \ref{thm:slow-mixing}.
\end{proof}

\subsection{\textsf{The lazy simple Markov chain of monotone paths in a strip}}

The lazy simple Markov chain $N_{m,n}$ on monotone paths starts at an initial state $Y_0$, and proceeds to state $Y_{t+1}$ by uniformly at random performing one of the local moves that are available at $Y_t$.
This Markov chain has period $2$, because the parity of the number of vertical steps changes with every local move. To make it aperiodic, we apply the usual strategy: slow down the walk with probability $1-p$ at each step.

\begin{definition} \label{robotic-arm-markov-chain}
\textbf{(The lazy simple Markov chain $N^p_{m,n}$ on monotone paths)} 
Let $0 < p < 1$. Let  $\Omega_{m,n}$ denote the set of monotone paths of length $n$ in a strip of height $m$. 
For $t \geq 0$, if $Y_t \in \Omega_{m,n}$ is the position of the path at time $t$, we obtain $Y_{t+1}$ as follows:
\begin{enumerate}
\item
With probability $(1-p)$, do nothing: let $Y_{t+1}=Y_t$. 
%\item
%Uniformly at random, either choose the $j$th interior joint between the $j$th and $j+1$st link for $1 \le j \le n-1$ or the endpoint of the arm, so that each possibility has probability $\frac1n$. 
\item 
With probability $p$ obtain $Y_{t+1}$ by uniformly at random performing one of the local moves that are available at $Y_t$.
\end{enumerate}
\end{definition}

\begin{theorem}\label{slow-mixing-theorem2}
Let $m\geq 2$ be a fixed integer. Let $N^p_{m,n}$ be the lazy simple Markov chain on monotone paths of length $n$ in a strip of height $m$.
The stationary distribution of $N^p_{m,n}$ is given by
\[
\pi(x) \propto \deg x
\]
for each path $x$, where $\deg x$ is the number of different local moves available at $x$. %, and $|E|$ is the set of all local moves among all states of the Markov chain. 

The mixing time of the chain $N^p_{m,n}$ grows exponentially with $n$; explicitly, there exists a constant $C>0$ such that for $n$ sufficiently large, the mixing time $\tau(\varepsilon)$ satisfies
\[
\tau(\varepsilon)  \ge C \cdot \frac1n r_m^n \ln(\varepsilon^{-1})
\]
where $r_m > 1$ is the Perron eigenvalue  of $A_m$ in Lemma \ref{lem:Perron}.
\end{theorem}

\begin{proof} The chain $N_{m,n}$ is irreducible for the same reasons as the chain $M_{m,n}$. It is aperiodic by construction. It can be verified that the chain is reversible with respect to the distribution $\pi$ since for each pair of states $x \ne y$, $P(x,y)= p/\deg(x)$. Hence, $\pi$ is the unique stationary distribution of the chain.
%
%
%
%
%The formula for the (lazy) stationary distribution of a simple Markov chain is well-known. For instance,  \cite[Example 1.12]{Levin09} contains a proof for the simple Markov chain, which is readily modified to contemplate the lazy walk as well. 

For the mixing time, we modify the proof in  Theorem \ref{slow-mixing-theorem}. 
Following the same notation, we have
\begin{eqnarray*}
\Phi & \le & \displaystyle \frac{\sum\limits_{x \in \S_b, y \notin \S_b} \pi(x)P(x,y)}{\pi(\S_b)} \\
& = & \displaystyle \frac{\sum\limits_{x \in \S_b, y \in \S_{a \nleftrightarrow b}}  \deg(x) P(x,y)}{\sum\limits_{x \in \S_b} \deg(x)} \\
& = & \displaystyle \frac{\sum\limits_{x \in \S_b, y \in \S_{a \nleftrightarrow b}, xy \in E(S_{m,n})} p}{\sum\limits_{x \in \S_b} \deg(x)}
\end{eqnarray*}
since $P(x,y) = p(1/\deg(x))$ whenever $x$ and $y$ are adjacent in the graph $S_{m,n}$, by the definition of the Markov chain. 

Again, there are exactly $n+1$ vertices in $\S_{a \nleftrightarrow b}$; let $Y_0, \ldots, Y_n$ be the corresponding order ideals of the coral PIP $C_{m,n}$, where $Y_i$ is a chain of length $i$. Notice that $Y_i$ has exactly one neighbor in $\S_b$ -- namely $X_i=Y_i \cup \{b\}$ -- for $i=2, \ldots, n$, and none for $i=0, 1$. Also $\deg(x) \geq 1$ for all $x$. This implies
\begin{eqnarray*}
\Phi & \leq & \frac{p(n-1)}{|\S_b|}  \\ 
& \sim & \frac{pn}{C_m q_m \cdot r_m^n}
\end{eqnarray*}
by Lemma \ref{lemma:bottleneck} and Corollary \ref{cor:c_mPerron}. 
The desired bound on the mixing time follows by Theorem \ref{thm:slow-mixing}.
\end{proof}

\section{\textsf{Further Directions}}

%\subsection{\textsf{Open Problems}}
\begin{itemize}

\item For the first few values of $m$, the polynomial $a_m(x)$ factor into two polynomials of degree $\lfloor \frac12(3m+1) \rfloor$ and $\lceil \frac12(3m+1) \rceil$ that are irreducible over $\mathbb{Z}$. Is this always the case? Which factor contributes the largest real root $r_m$? Is $r_m$ always inexpressible in terms of radicals?

\item 
The denominator $\det(I-xA_m) = (-x)^{3m+1}a_m(-1/x)$ in Theorem \ref{thm:detformula} is essentially equal to the characteristic polynomial $a_m(x)$, for which we found an explicit formula in Proposition \ref{prop:charpoly}. Can we also find an explicit formula for the numerator?

\item
Can we generalize the results in this paper to monotone paths in wedges, as studied by Janse van Rensburg, Prellberg, and Rechnitzer\cite{JPR}?

%\todo{Try!}

\item The authors of  \cite{Ard17}, ask whether the configuration space of not necessarily monotone, but still self-avoiding paths in a strip of height $m$ is still CAT(0). In our context, we can ask about the growth rate constants -- which have been computed for $m \leq 2$ by Dangovski and Lalov \cite{DL} -- 
and the mixing times of the corresponding Markov chains.

\item Is there a natural set of moves that connects the monotone paths of length $n$ in a strip of height $m$, for which the Markov chain mixes in polynomial time?

\end{itemize}

\section{\textsf{Acknowledgements}}
Coleson and Naya thank the University of Delaware and the Department of Mathematical Sciences for making this collaboration possible. Their work was supported by National Science Foundation grant DMS-1554783. Coleson would like to thank Naya for allowing him the opportunity to conduct research with her.
Federico was supported by National Science Foundation grant DMS-2154279. He would like to thank Mariana Smit Vega for help with the proof of Proposition \ref{prop:Perron}. He is also grateful to the announcer of KRZZ La Raza 93.3 who  serendipitously announced ``y ahora, un corrido bien perr\'on" and made him realize that Perron eigenvalues would play an important role in this project.

\bibliographystyle{alpha}
\bibliography{monotonepathsmixslowly}

\end{document}

%% file: G2.pdf_t
\begin{picture}(0,0)%
\includegraphics{G2.pdf}%
\end{picture}%
\setlength{\unitlength}{3947sp}%
\begingroup\makeatletter\ifx\SetFigFont\undefined%
\gdef\SetFigFont#1#2#3#4#5{%
  \reset@font\fontsize{#1}{#2pt}%
  \fontfamily{#3}\fontseries{#4}\fontshape{#5}%
  \selectfont}%
\fi\endgroup%
\begin{picture}(6279,3487)(3790,-5795)
\put(5326,-2611){\makebox(0,0)[lb]{\smash{{\SetFigFont{20}{24.0}{\familydefault}{\mddefault}{\updefault}{\color[rgb]{0,0,0}$(0,1)$}%
}}}}
\put(7201,-4186){\makebox(0,0)[lb]{\smash{{\SetFigFont{20}{24.0}{\familydefault}{\mddefault}{\updefault}{\color[rgb]{0,0,0}$(1,1)$}%
}}}}
\put(8476,-4186){\makebox(0,0)[lb]{\smash{{\SetFigFont{20}{24.0}{\familydefault}{\mddefault}{\updefault}{\color[rgb]{0,0,0}$(2,2)$}%
}}}}
\put(7801,-5686){\makebox(0,0)[lb]{\smash{{\SetFigFont{20}{24.0}{\familydefault}{\mddefault}{\updefault}{\color[rgb]{0,0,0}$(2,1)$}%
}}}}
\put(7801,-2611){\makebox(0,0)[lb]{\smash{{\SetFigFont{20}{24.0}{\familydefault}{\mddefault}{\updefault}{\color[rgb]{0,0,0}$(1,2)$}%
}}}}
\put(5326,-5686){\makebox(0,0)[lb]{\smash{{\SetFigFont{20}{24.0}{\familydefault}{\mddefault}{\updefault}{\color[rgb]{0,0,0}$(1,0)$}%
}}}}
\put(4726,-4186){\makebox(0,0)[lb]{\smash{{\SetFigFont{20}{24.0}{\familydefault}{\mddefault}{\updefault}{\color[rgb]{0,0,0}$(0,0)$}%
}}}}
\end{picture}%